\documentclass[12pt,letterpaper,leqno]{amsart}
\usepackage{microtype}
\usepackage{amssymb}
\usepackage{amsfonts}
\usepackage{geometry}

\newcommand{\indentalign}{\hspace{0.3in}&\hspace{-0.3in}}

\newtheorem{theorem}{Theorem}
\theoremstyle{plain}

\newtheorem{conjecture}{Conjecture}
\newtheorem{corollary}{Corollary}

\newtheorem{lemma}{Lemma}

\numberwithin{equation}{section}
\numberwithin{theorem}{section}  
\numberwithin{proposition}{section}  
\numberwithin{lemma}{section}  
\numberwithin{corollary}{section}  
\input{tcilatex}
\begin{document}
\title[Proof of the Klainerman-Machedon conjecture with high $\beta $]{On
the Klainerman-Machedon conjecture for the quantum BBGKY hierarchy with
self-interaction}
\author{Xuwen Chen}
\address{Department of Mathematics, Brown University, 151 Thayer Street,
Providence, RI 02912}
\email{chenxuwen@math.brown.edu}
\urladdr{http://www.math.brown.edu/\symbol{126}chenxuwen/}
\author{Justin Holmer}
\address{Department of Mathematics, Brown University, 151 Thayer Street,
Providence, RI 02912}
\email{holmer@math.brown.edu}
\urladdr{http://www.math.brown.edu/\symbol{126}holmer/}
\date{V2 for JEMS, 02/04/2014}

\begin{abstract}
We consider the 3D quantum BBGKY hierarchy which corresponds to the $N$%
-particle Schr\"{o}dinger equation. We assume the pair interaction is $%
N^{3\beta -1}V(N^{\beta }\bullet )$. For interaction parameter $\beta \in (0,%
\frac{2}{3})$, we prove that, provided an energy bound holds for solutions
to the BBKGY hierarchy, the $N\rightarrow \infty $ limit points satisfy the
space-time bound conjectured by S. Klainerman and M. Machedon \cite%
{KlainermanAndMachedon} in 2008. The energy bound is proven to hold for $%
\beta \in (0,\frac{3}{5})$ in \cite{E-S-Y2}. This allows, in the case $\beta
\in (0,\frac{3}{5})$, for the application of the Klainerman-Machedon
uniqueness theorem and hence implies that the $N\rightarrow \infty $ limit
of BBGKY is uniquely determined as a tensor product of solutions to the
Gross-Pitaevskii equation when the $N$-body initial data is factorized. The
first result in this direction in 3D was obtained by T. Chen and N. Pavlovi%
\'{c} \cite{TChenAndNPSpace-Time} for $\beta \in (0,\frac{1}{4})$ and
subsequently by X. Chen \cite{Chen3DDerivation} for $\beta \in (0,\frac{2}{7}%
]$. We build upon the approach of X. Chen but apply frequency localized
Klainerman-Machedon collapsing estimates and the endpoint Strichartz
estimate in the estimate of the \textquotedblleft potential
part\textquotedblright\ to extend the range to $\beta \in (0,\frac{2}{3})$.
Overall, this provides an alternative approach to the mean-field program by
L. Erd\"{o}s, B. Schlein, and H.-T. Yau \cite{E-S-Y2}, whose uniqueness
proof is based upon Feynman diagram combinatorics.
\end{abstract}

\maketitle
\tableofcontents

%\date{01/18/2013}

\section{Introduction}

\label{S:Introduction}

The 3D quantum BBGKY (Bogoliubov-Born-Green-Kirkwood-Yvon) hierarchy is
generated from the $N$-body Hamiltonian evolution $\psi
_{N}(t)=e^{itH_{N}}\psi _{N,0}$ with symmetric initial datum and the $N$%
-body Hamiltonian is given by 
\begin{equation}
H_{N}=-\triangle _{\mathbf{x}_{N}}+\frac{1}{N}\sum_{1\leqslant i<j\leqslant
N}N^{3\beta }V(N^{\beta }\left( x_{i}-x_{j}\right) ).  \label{def:H_N}
\end{equation}%
In the above, $t\in \mathbb{R}$, $\mathbf{x}_{N}=(x_{1},x_{2},\cdots
,x_{N})\in \mathbb{R}^{3N}$, $\triangle _{\mathbf{x}_{N}}$ denotes the
standard Laplacian with respect to the variables $\mathbf{x}_{N}\in \mathbb{R%
}^{3N}$, the factor $1/N$ in (\ref{def:H_N}) is to make sure that the
interactions are proportional to the number of particles, and the pair
interaction $N^{3\beta }V(N^{\beta }\left( x_{i}-x_{j}\right) )$ is an
approximation to the Dirac $\delta $ function which matches the
Gross-Pitaevskii description of Bose-Einstein condensation that the
many-body effect should be modeled by a strong on-site self-interaction.
Since $\psi _{N}\overline{\psi _{N}}$ is a probability density, we define
the marginal densities $\left\{ \gamma _{N}^{(k)}(t,\mathbf{x}_{k};\mathbf{x}%
_{k}^{\prime })\right\} _{k=1}^{N}$ by 
\begin{equation*}
\gamma _{N}^{(k)}(t,\mathbf{x}_{k};\mathbf{x}_{k}^{\prime })=\int \psi
_{N}(t,\mathbf{x}_{k},\mathbf{x}_{N-k})\overline{\psi _{N}}(t,\mathbf{x}%
_{k}^{\prime },\mathbf{x}_{N-k})d\mathbf{x}_{N-k},\text{ }\mathbf{x}_{k},%
\mathbf{x}_{k}^{\prime }\in \mathbb{R}^{3k}.
\end{equation*}%
Then we have that $\left\{ \gamma _{N}^{(k)}(t,\mathbf{x}_{k};\mathbf{x}%
_{k}^{\prime })\right\} _{k=1}^{N}$ is a sequence of trace class operator
kernels which are symmetric, in the sense that 
\begin{equation*}
\gamma _{N}^{(k)}(t,\mathbf{x}_{k},\mathbf{x}_{k}^{\prime })=\overline{%
\gamma _{N}^{(k)}(t,\mathbf{x}_{k}^{\prime },\mathbf{x}_{k})},
\end{equation*}%
and 
\begin{equation}
\gamma _{N}^{(k)}(t,x_{\sigma (1)},\cdots x_{\sigma (k)},x_{\sigma
(1)}^{\prime },\cdots x_{\sigma (k)}^{\prime })=\gamma
_{N}^{(k)}(t,x_{1},\cdots ,x_{k},x_{1}^{\prime },\cdots ,x_{k}^{\prime }),
\label{condition:symmetric}
\end{equation}%
for any permutation $\sigma ,$ and satisfy the 3D quantum BBGKY hierarchy of
equations which written in operator form is 
\begin{equation}
i\partial _{t}\gamma _{N}^{(k)}+\left[ \triangle _{\mathbf{x}_{k}},\gamma
_{N}^{(k)}\right] =\begin{aligned}[t] &\frac{1}{N}\sum_{1\leqslant
i<j\leqslant k}\left[ V_{N}\left( x_{i}-x_{j}\right) ,\gamma
_{N}^{(k)}\right] \\ &+\frac{N-k}{N}
\sum_{j=1}^{k}\limfunc{Tr}\nolimits_{k+1}\left[ V_{N}\left(
x_{j}-x_{k+1}\right) ,\gamma _{N}^{(k+1)}\right] \end{aligned}
\label{hierarchy:BBGKY hierarchy in operator form}
\end{equation}%
if we do not distinguish $\gamma _{N}^{(k)}$ as a kernel and the operator it
defines.\footnote{%
From here on out, we consider only the $\beta >0$ case. For $\beta =0,$ see 
\cite%
{E-Y1,Frolich,KnowlesAndPickl,MichelangeliSchlein,RodnianskiAndSchlein,GMM1,GMM2,Chen2ndOrder, LChen}
.} Here the operator $V_{N}\left( x\right) $ represents multiplication by
the function $V_{N}\left( x\right) $, where 
\begin{equation}
V_{N}\left( x\right) =N^{3\beta }V(N^{\beta }x),  \label{def:V_N}
\end{equation}%
and $\limfunc{Tr}\nolimits_{k+1}$ means taking the $k+1$ trace, for example, 
\begin{equation*}
\limfunc{Tr}\nolimits_{k+1}V_{N}\left( x_{j}-x_{k+1}\right) \gamma
_{N}^{(k+1)}=\int V_{N}\left( x_{j}-x_{k+1}\right) \gamma _{N}^{(k+1)}(t,%
\mathbf{x}_{k},x_{k+1};\mathbf{x}_{k}^{\prime },x_{k+1})dx_{k+1}.
\end{equation*}

In 2008, S. Klainerman and M. Machedon implicitly made the following
conjecture on the solution of the BBGKY hierarchy.

\begin{conjecture}[Klainerman-Machedon \protect\cite{KlainermanAndMachedon}]

\label{Conjecture:KM}Assume the interaction parameter $\beta \in \left( 0,1%
\right] $. Suppose that the sequence $\left\{ \gamma _{N}^{(k)}(t,\mathbf{x}%
_{k};\mathbf{x}_{k}^{\prime })\right\} _{k=1}^{N}$ is a solution to the 3D
quantum BBGKY hierarchy (\ref{hierarchy:BBGKY hierarchy in operator form})
subject to the energy condition: there is a $C_{0}$ (independent of $N$ and $%
k$) such that for any $k\geqslant 0,$ there is a $N_{0}(k)$ such that 
\begin{equation}
\forall \;N\geqslant N_{0}(k)\,,\qquad \sup_{t\in \mathbb{R}}\limfunc{Tr}%
\left( \dprod\limits_{j=1}^{k}\left( 1-\triangle _{x_{j}}\right) \right)
\gamma _{N}^{(k)}\leqslant C_{0}^{k}.  \label{condition:energy condition Tr}
\end{equation}%
Then, for every finite time $T$, every limit point $\Gamma =\left\{ \gamma
^{(k)}\right\} _{k=1}^{\infty }$ of $\left\{ \Gamma _{N}\right\}
_{N=1}^{\infty }=\left\{ \left\{ \gamma _{N}^{(k)}\right\}
_{k=1}^{N}\right\} _{N=1}^{\infty }$ in $\bigoplus_{k\geqslant 1}C\left( %
\left[ 0,T\right] ,\mathcal{L}_{k}^{1}\right) $ with respect to the product
topology $\tau _{prod}$ (defined in Appendix \ref{appendix:ESYTopology})
satisfies the space-time bound: there is a $C$ independent of $j,k$ such
that 
\begin{equation}
\int_{0}^{T}\left\Vert R^{(k)}B_{j,k+1}\gamma ^{(k+1)}\left( t\right)
\right\Vert _{L_{\mathbf{x},\mathbf{x}^{\prime }}^{2}}dt\leqslant C^{k},
\label{bound:KMConjecture}
\end{equation}%
where $\mathcal{L}_{k}^{1}$ is the space of trace class operators on $L^{2}(%
\mathbb{R}^{3k}),$ $R^{(k)}=\dprod\limits_{j=1}^{k}\left( \left\vert \nabla
_{x_{j}}\right\vert \left\vert \nabla _{x_{j}^{\prime }}\right\vert \right)
, $ and 
\begin{equation*}
B_{j,k+1}=\limfunc{Tr}\nolimits_{k+1}\left[ \delta \left(
x_{j}-x_{k+1}\right) ,\gamma ^{(k+1)}\right] .
\end{equation*}
\end{conjecture}

Though Conjecture \ref{Conjecture:KM} was not stated explicitly in \cite%
{KlainermanAndMachedon}, as we will explain after stating Theorem \ref%
{THM:MainTHM}, the bound (\ref{bound:KMConjecture}) is necessary to
implement Klainerman-Machedon's powerful and flexible approach in the most
involved part of the quantum Kac's program which mathematically proves the
cubic nonlinear Schr\"{o}dinger equation (NLS) as the $N\rightarrow \infty $
limit of quantum $N$-body dynamics. Kirkpatrick-Schlein-Staffilani \cite%
{Kirpatrick} completely solved the $\mathbb{T}^{2}$ version of Conjecture %
\ref{Conjecture:KM} and were the first to successfully implement such an
approach. However, Conjecture \ref{Conjecture:KM}, the $\mathbb{R}^{3}$
version as stated, was fully open until recently. T. Chen and Pavlovi\'{c} 
\cite{TChenAndNPSpace-Time} have been able to prove Conjecture \ref%
{Conjecture:KM} for $\beta \in \left( 0,1/4\right) $. In \cite%
{Chen3DDerivation}, X.C simplified and extended the result to the range of $%
\beta \in $ $\left( 0,2/7\right] .$ We devote this paper to proving
Conjecture \ref{Conjecture:KM} for $\beta \in \left( 0,2/3\right) $. In
particular, we surpass the self-interaction thresold$\footnote{%
We will explain why we call the $\beta >1/3$ case self-interaction later in
this introduction.}$, namely $\beta =1/3$. To be specific, we prove the
following theorem.

\begin{theorem}[Main theorem]
\label{THM:MainTHM}Assume the interaction parameter $\beta \in \left(
0,2/3\right) $ and the pair interaction $V\in L^{1}\cap W^{2,\frac{6}{5}+}.$
Under condition (\ref{condition:energy condition}), every limit point $%
\Gamma =\left\{ \gamma ^{(k)}\right\} _{k=1}^{\infty }$ of $\left\{ \Gamma
_{N}\right\} _{N=1}^{\infty }$ satisfies the Klainerman-Machedon space-time
bound (\ref{bound:KMConjecture}).
\end{theorem}

Establishing the $N\rightarrow \infty $ limit of hierarchy (\ref%
{hierarchy:BBGKY hierarchy in operator form}) justifies the mean-field limit
in the Gross-Pitaevskii theory. Such an approach was first proposed by Spohn 
\cite{Spohn}\ and can be regarded as a quantum version of Kac's program. We
see that, as $N\rightarrow \infty ,$ hierarchy (\ref{hierarchy:BBGKY
hierarchy in operator form}) formally converges to the infinite
Gross-Pitaevskii hierarchy 
\begin{equation}
i\partial _{t}\gamma ^{(k)}+\left[ \triangle _{\mathbf{x}_{k}},\gamma ^{(k)}%
\right] =\left( \int V(x)dx\right) \sum_{j=1}^{k}\limfunc{Tr}\nolimits_{k+1}%
\left[ \delta \left( x_{j}-x_{k+1}\right) ,\gamma ^{(k+1)}\right] .
\label{hierarchy:GP hierarchy}
\end{equation}%
When the initial data is factorized 
\begin{equation*}
\gamma ^{(k)}(0,\mathbf{x}_{k};\mathbf{x}_{k}^{\prime
})=\dprod\limits_{j=1}^{k}\phi _{0}(x_{j})\bar{\phi}_{0}(x_{j}),
\end{equation*}%
hierarchy \eqref{hierarchy:GP hierarchy} has a special solution 
\begin{equation}
\gamma ^{(k)}(t,\mathbf{x}_{k};\mathbf{x}_{k}^{\prime
})=\dprod\limits_{j=1}^{k}\phi (t,x_{j})\bar{\phi}(t,x_{j}),
\label{eqn:product form}
\end{equation}%
if $\phi $ solves the cubic NLS 
\begin{equation}
i\partial _{t}\phi =-\triangle _{x}\phi +\left( \int V(x)dx\right)
\left\vert \phi \right\vert ^{2}\phi .  \label{eqn:cubic NLS}
\end{equation}%
Thus such a limit process shows that, in an appropriate sense, 
\begin{equation*}
\lim_{N\rightarrow \infty }\gamma _{N}^{(k)}=\dprod\limits_{j=1}^{k}\phi
(t,x_{j})\bar{\phi}(t,x_{j}),
\end{equation*}%
hence justifies the mean-field limit.

Such a limit in 3D was first proved in a series of important papers \cite%
{E-E-S-Y1, E-S-Y1,E-S-Y2,E-S-Y5, E-S-Y3} by Elgart, Erd\"{o}s, Schlein, and
Yau.\footnote{%
Around the same time, there was the 1D work \cite{AGT}.} Briefly, the
Elgart-Erd\"{o}s-Schlein-Yau approach\footnote{%
See \cite{SchleinNew,GM1,Pickl} for different approaches.} can be described
as the following:

\bigskip

\noindent \emph{Step A}. Prove that, with respect to the topology $\tau
_{prod}$ defined in Appendix \ref{appendix:ESYTopology}, the sequence $%
\left\{ \Gamma _{N}\right\} _{N=1}^{\infty }$ is compact in the space $%
\bigoplus_{k\geqslant 1}C\left( \left[ 0,T\right] ,\mathcal{L}^{1}\left( 
\mathbb{R}^{3k}\right) \right) )$.

\bigskip

\noindent \emph{Step B}. Prove that every limit point $\Gamma =\left\{
\gamma ^{(k)}\right\} _{k=1}^{\infty }$ of $\left\{ \Gamma _{N}\right\}
_{N=1}^{\infty }$ must verify hierarchy \eqref{hierarchy:GP hierarchy}.

\bigskip

\noindent \emph{Step C}. Prove that, in the space in which the limit points
from Step B lie, there is a unique solution to hierarchy (\ref{hierarchy:GP
hierarchy}). Thus $\{\Gamma_N\}_{N=1}^\infty$ is a compact sequence with
only one limit point. Hence $\Gamma_N \to \Gamma$ as $N\to \infty$.

\bigskip

In 2007, Erd\"{o}s, Schlein, and Yau obtained the first uniqueness theorem
of solutions \cite[Theorem 9.1]{E-S-Y2} to the hierarchy 
\eqref{hierarchy:GP
hierarchy}. The proof is surprisingly delicate -- it spans 63 pages and uses
complicated Feynman diagram techniques. The main difficulty is that
hierarchy \eqref{hierarchy:GP hierarchy} is a system of infinitely coupled
equations. Briefly, \cite[Theorem 9.1]{E-S-Y2} is the following:

\begin{theorem}[{Erd\"os-Schlein-Yau uniqueness {\protect\cite[Theorem 9.1]%
{E-S-Y2}}}]
There is at most one nonnegative symmetric operator sequence $\left\{ \gamma
^{(k)}\right\} _{k=1}^{\infty }$ that solves hierarchy 
\eqref{hierarchy:GP
hierarchy} subject to the energy condition 
\begin{equation}
\sup_{t\in \lbrack 0,T]}\limfunc{Tr}\left( \dprod\limits_{j=1}^{k}\left(
1-\triangle _{x_{j}}\right) \right) \gamma ^{(k)}\leqslant C^{k}.
\label{bound:ESYCondition}
\end{equation}
\end{theorem}

In \cite{KlainermanAndMachedon}, based on their null form paper \cite%
{KlainermanMachedonNullForm}, Klainerman and Machedon gave a different proof
of the uniqueness of hierarchy (\ref{hierarchy:GP hierarchy}) in a space
different from that used in \cite[Theorem 9.1]{E-S-Y2}. The proof is shorter
(13 pages) than the proof of \cite[Theorem 9.1]{E-S-Y2}. Briefly, \cite[%
Theorem 1.1]{KlainermanAndMachedon} is the following:

\begin{theorem}[{Klainerman-Machedon uniqueness {\protect\cite[Theorem 1.1]%
{KlainermanAndMachedon}}}]
There is at most one symmetric operator sequence $\left\{ \gamma
^{(k)}\right\} _{k=1}^{\infty }$ that solves hierarchy 
\eqref{hierarchy:GP
hierarchy} subject to the space-time bound \eqref{bound:KMConjecture}.
\end{theorem}

For special cases like (\ref{eqn:product form}), condition (\ref%
{bound:ESYCondition}) is actually 
\begin{equation}
\sup_{t\in \lbrack 0,T]}\left\Vert \left\langle \nabla _{x}\right\rangle
\phi \right\Vert _{L^{2}}\leqslant C,  \label{estimate:ESYforNLS}
\end{equation}%
while condition (\ref{bound:KMConjecture}) means 
\begin{equation}
\int_{0}^{T}\left\Vert \left\vert \nabla _{x}\right\vert \left( \left\vert
\phi \right\vert ^{2}\phi \right) \right\Vert _{L^{2}}dt\leqslant C.
\label{estimate:KMforNLS}
\end{equation}%
When $\phi $ satisfies NLS (\ref{eqn:cubic NLS}), both are known. In fact,
due to the Strichartz estimate \cite{Keel-Tao}, (\ref{estimate:ESYforNLS})
implies (\ref{estimate:KMforNLS}), that is, condition (\ref%
{bound:KMConjecture}) seems to be a bit weaker than condition (\ref%
{bound:ESYCondition}). The proof of \cite[Theorem 1.1]{KlainermanAndMachedon}
(13 pages) is also considerably shorter than the proof of \cite[Theorem 9.1]%
{E-S-Y2} (63 pages). It is then natural to wonder whether \cite[Theorem 1.1]%
{KlainermanAndMachedon} simplifies Step C. To answer such a question it is
necessary to know whether the limit points in Step B satisfy condition %
\eqref{bound:KMConjecture}, that is, whether Conjecture \ref{Conjecture:KM}
holds.

Away from curiosity, there are realistic reasons to study Conjecture \ref%
{Conjecture:KM}. While \cite[Theorem 9.1]{E-S-Y2} is a powerful theorem, it
is very difficult to adapt such an argument to various other interesting and
colorful settings: a different spatial dimension, a three-body interaction
instead of a pair interaction, or the Hermite operator instead of the
Laplacian. The last situation mentioned is physically important. On the one
hand, all the known experiments of BEC use harmonic trapping to stabilize
the condensate \cite{Anderson, Davis,Cornish, Ketterle, Stamper}. On the
other hand, different trapping strength produces quantum behaviors which do
not exist in the Boltzmann limit of classical particles nor in the quantum
case when the trapping is missing and have been experimentally observed \cite%
{Kettle3Dto2DExperiment, FrenchExperiment, Philips, NatureExperiment,
Another2DExperiment}. The Klainerman-Machedon approach applies easily in
these meaningful situations (\cite%
{Kirpatrick,TChenAndNP,ChenAnisotropic,Chen3DDerivation,C-H3Dto2D,Sohinger}%
). Thus proving Conjecture \ref{Conjecture:KM} actually helps to advance the
study of quantum many-body dynamic and the mean-field approximation in the
sense that it provides a flexible and powerful tool in 3D.

The well-posedness theory of the Gross-Pitaevskii hierarchy (\ref%
{hierarchy:GP hierarchy}) subject to general initial datum also requires
that the limits of the BBGKY hierarchy (\ref{hierarchy:BBGKY hierarchy in
operator form}) lie in the space in which the space-time bound (\ref%
{bound:KMConjecture}) holds. See \cite%
{TChenAndNPGPWellPosedness1,TChenAndNPGPWellPosedness2,TChenAndNPSpace-Time}.

As pointed out in \cite{E-E-S-Y1}, the study of the Hamiltonian (\ref%
{def:H_N}) is of particular interest when $\beta \in \left( 1/3,1\right] $.
The reason is the following. In physics, the initial datum $\psi _{N}\left(
0\right) $ of the Hamiltonian evolution $e^{itH_{N}}\psi _{N}\left( 0\right) 
$ is usually assumed to be close to the ground state of the Hamiltonian 
\begin{equation*}
H_{N,0}=-\triangle _{\mathbf{x}_{N}}+\omega _{0}^{2}\left\vert \mathbf{x}%
_{N}\right\vert ^{2}+\frac{1}{N}\sum_{1\leqslant i<j\leqslant N}N^{3\beta
}V(N^{\beta }\left( x_{i}-x_{j}\right) ).
\end{equation*}%
The preparation of the available experiments and the mathematical work \cite%
{Lieb2} by Lieb, Seiringer, Solovej and Yngvason confirm this assumption.
Such an initial datum $\psi _{N}\left( 0\right) $ is localized in space. We
can assume all $N$ particles are in a box of length $1$. Let the effective
radius of the pair interaction $V$ be $a,$ then the effective radius of $%
V_{N}$ is about $a/N^{\beta }$. Thus every particle in the box interacts
with $\left( a/N^{\beta }\right) ^{3}\times N$ other particles. Thus, for $%
\beta >1/3$ and large $N$, every particle interacts with only itself. This
exactly matches the Gross-Pitaevskii theory that the many-body effect should
be modeled by a strong on-site self-interaction. Therefore, for the
mathematical justification of the Gross-Pitaevskii theory, it is of
particular interest to prove Conjecture \ref{Conjecture:KM} for
self-interaction ($\beta >1/3)$ as well.

To the best of our knowledge, the main theorem (Theorem \ref{THM:MainTHM})
in the current paper is the first result in proving Conjecture \ref%
{Conjecture:KM} for self-interaction ($\beta >1/3).$ For $\beta \leqslant
1/3 $, the first progress of Conjecture \ref{Conjecture:KM} is the $\beta
\in (0,1/4)$ work \cite{TChenAndNPSpace-Time}\ by T. Chen and N. Pavlovi\'{c}
and then the $\beta \in (0,2/7]$ work \cite{Chen3DDerivation}\ by X.C. As a
matter of fact, the main theorem (Theorem \ref{THM:MainTHM}) in the current
paper has already fulfilled the original intent of \cite%
{KlainermanAndMachedon}, namely, simplifying the uniqueness argument of \cite%
{E-S-Y2}, because \cite{E-S-Y2} deals with $\beta \in (0,3/5).$ Conjecture %
\ref{Conjecture:KM} for $\beta \in \lbrack 2/3,1]$ is still open.

\subsection{Organization of the paper}

In \S \ref{S:proof-main-theorem}, we outline the proof of Theorem \ref%
{THM:MainTHM}. The overall pattern follows that introduced by X.C. \cite%
{Chen3DDerivation}, who obtained Theorem \ref{THM:MainTHM} for $\beta \in (0,%
\frac{2}{7}]$. Let $P_{\leq M}^{(k)}$ be the Littlewood-Paley projection
defined in \eqref{E:LP-def}. Theorem \ref{THM:MainTHM} will follow once it
is established that for all $M\geq 1$, there exists $N_{0}$ depending on $M$
such that for all $N\geq N_{0}$, there holds 
\begin{equation}
\Vert P_{\leq M}^{(k)}R^{(k)}B_{N,j,k+1}\gamma _{N}^{(k+1)}(t)\Vert
_{L_{T}^{1}L_{\mathbf{x},\mathbf{x}^{\prime }}^{2}}\leqslant C^{k}
\label{E:main-bound1}
\end{equation}%
where $B_{N,j,k+1}$ is defined by \eqref{E:BN-def}. Substituting the
Duhamel-Born expansion, carried out to coupling level $l_{c}$, of the BBGKY
hierarchy, this is reduced to proving analogous bounds on the free part,
potential part, and interaction part, defined in \S \ref%
{S:proof-main-theorem}. Each part is reduced via the Klainerman-Machedon
board game. Estimates for the free part and interaction part were previously
obtained by X.C. \cite{Chen3DDerivation} but are reproduced here for
convenience in Appendix \ref{appendix:estimating free and interaction}. For
the estimate of the interaction part, one takes $l_{c}=\ln N$, the utility
of which was first observed by T. Chen and N. Pavlovi\'{c} \cite%
{TChenAndNPSpace-Time}.

The main new achievement of our paper is the improved estimates on the
potential part, which are discussed in \S \ref{Sec:estimating potential part}%
. We make use of the endpoint Strichartz estimate, phrased in terms of the $%
X_b$ norm, in place of the Sobolev inequality employed by X.C \cite%
{Chen3DDerivation}. We also introduce frequency localized versions of the
Klainerman-Machedon collapsing estimates, allowing us to exploit the
frequency localization in \eqref{E:main-bound1}. Specifically, the operator $%
P^{(k)}_{\leq M}$ does not commute with $B_{N,j,k+1}$, however, the
composition $P^{(k)}_{\leq M_k} B_{N,j,k+1} P^{(k+1)}_{\sim M_{k+1}}$ enjoys
better bounds if $M_{k+1} \gg M_k$. We prove the Strichartz estimate and the
frequency localized Klainerman-Machedon collapsing estimates in \S \ref%
{S:Strichartz-X}. Frequency localized space-time techniques of this type
were introduced by Bourgain \cite[Chapter IV, \S 3]{Bourgain} into the study
of the well-posedness for nonlinear Schr\"odinger equations and other
nonlinear dispersive PDE.

In X.C. \cite{Chen3DDerivation}, \eqref{E:main-bound1} is obtained without
the frequency localization $P^{(k)}_{\leq M}$ for $\beta \in (0,\frac27]$.
In Theorem \ref{Theorem:example:beta up to 2/5}, we prove that this estimate
still holds without frequency localization for $\beta \in (0,\frac25)$ by
using the Strichartz estimate alone. This already surpasses the
self-interaction threshold $\beta = \frac13$. For the purpose of proving
Conjecture \ref{Conjecture:KM}, the frequency localized estimate %
\eqref{E:main-bound1} is equally good, but allows us to achieve higher $%
\beta $.

\subsection{Acknowledgements}

J.H. was supported in part by NSF grant DMS-0901582 and a Sloan Research
Fellowship (BR-4919), and X.C. received travel support from the same Sloan
Fellowship to visit U. Maryland. We would like to thank T. Chen, M.
Grillakis, M. Machedon, and N. Pavlovi\'{c} for very helpful discussions
related to this work, and we would like to thank the anonymous referee for
many helpful suggestions.

\section{Proof of the main theorem}

\label{S:proof-main-theorem}

We establish Theorem \ref{THM:MainTHM} in this section. For simplicity of
notation, we denote $\left\Vert \cdot \right\Vert_{L^{p}\left[ 0,T\right]L_{%
\mathbf{x,x}^{\prime }}^{2}}$ by $\left\Vert \cdot \right\Vert_{L_{T}^{p}L_{%
\mathbf{x,x}^{\prime }}^{2}}$ and denote $\left\Vert \cdot \right\Vert
_{L_{t}^{p}\left( \mathbb{R}\right) L_{\mathbf{x,x}^{\prime}}^{2}}$ by $%
\left\Vert \cdot \right\Vert _{L_{t}^{p}L_{\mathbf{x,x}^{\prime }}^{2}}$.
Let us begin by introducing some notation for Littlewood-Paley theory. Let $%
P_{\leq M}^{i}$ be the projection onto frequencies $\leq M$ and $P_{M}^{i}$
the analogous projections onto frequencies $\sim M$, acting on functions of $%
x_{i}\in \mathbb{R}^{3}$ (the $i$th coordinate). We take $M$ to be a dyadic
frequency range $2^{\ell }\geq 1$. Similarly, we define $P_{\leq
M}^{i^{\prime }}$ and $P_{M}^{i^{\prime }}$, which act on the variable $%
x_{i}^{\prime }$. Let 
\begin{equation}  \label{E:LP-def}
P_{\leq M}^{(k)}=\prod_{i=1}^{k}P_{\leq M}^{i}P_{\leq M}^{i^{\prime }}.
\end{equation}
To establish Theorem \ref{THM:MainTHM}, it suffices to prove the following
theorem.

\begin{theorem}
\label{THM:LocalizedMainTHM}Under the assumptions of Theorem \ref%
{THM:MainTHM}, there exists a $C$ (independent of $k,M,N$) such that for
each $M\geq 1$ there exists $N_{0}$ (depending on $M$) such that for $%
N\geqslant N_{0}$, there holds 
\begin{equation}
\Vert P_{\leq M}^{(k)}R^{(k)}B_{N,j,k+1}\gamma _{N}^{(k+1)}(t)\Vert
_{L_{T}^{1}L_{\mathbf{x},\mathbf{x}^{\prime }}^{2}}\leqslant C^{k}
\label{estimate:main target}
\end{equation}
where 
\begin{equation}  \label{E:BN-def}
B_{N,j,k+1}\gamma _{N}^{(k+1)}=\limfunc{Tr}\nolimits_{k+1}\left[ V_{N}\left(
x_{j}-x_{k+1}\right) ,\gamma _{N}^{(k+1)}\right] .
\end{equation}
\end{theorem}

We first explain how, assuming Theorem \ref{THM:LocalizedMainTHM}, we can
prove Theorem \ref{THM:MainTHM}. When condition (\ref{condition:energy
condition Tr}) holds, it has been proved in Elgart-Erd\"{o}s-Schlein-Yau 
\cite{E-E-S-Y1, E-S-Y1,E-S-Y2,E-S-Y5, E-S-Y3} and
Kirkpatrick-Schlein-Staffilani \cite{Kirpatrick} that, as trace class
operators%
\begin{equation}
B_{N,j,k+1}\gamma _{N}^{(k+1)}\rightharpoonup B_{j,k+1}\gamma ^{(k+1)}\text{
(weak*),}  \label{formula:weak*}
\end{equation}%
uniformly in $t$. (See \cite[(6.7)]{Kirpatrick} or \cite[(5.6)]{C-H3Dto2D},
for example.) Let $\mathcal{H}_{k}$ be the Hilbert-Schmidt operators on $%
L^{2}\left( \mathbb{R}^{3k}\right) .$ Recall that the test functions for
weak* convergence in $\mathcal{L}_{k}^{1}$ come from $\mathcal{K}_{k}$ and
the test functions for weak* convergence in $\mathcal{H}_{k}$ come from $%
\mathcal{H}_{k}.$ Thus the weak* convergence (\ref{formula:weak*}) as trace
class operator infers that as Hilbert-Schmidt operators,%
\begin{equation*}
B_{N,j,k+1}\gamma _{N}^{(k+1)}\rightharpoonup B_{j,k+1}\gamma ^{(k+1)}\text{
(weak*),}
\end{equation*}%
uniformly in $t,$ because $\mathcal{H}_{k}\subset \mathcal{K}_{k}$ i.e.
there are fewer test functions. Since $\mathcal{H}_{k}$ is reflexive, the
above weak* convergence is no different from the weak convergence. Moreover,
notice that $P_{\leq M}^{(k)}R^{(k)}J$ is simply another test function if $J$
is a test function, we know that as Hilbert-Schmidt operators%
\begin{equation*}
P_{\leq M}^{(k)}R^{(k)}B_{N,j,k+1}\gamma _{N}^{(k+1)}\rightharpoonup P_{\leq
M}^{(k)}R^{(k)}B_{j,k+1}\gamma ^{(k+1)}\text{ (weak)}
\end{equation*}%
uniformly in $t.$ Hence, by basic properties of weak convergence 
\begin{equation*}
\Vert P_{\leq M}^{(k)}R^{(k)}B_{j,k+1}\gamma ^{(k+1)}\Vert _{L_{T}^{1}L_{%
\mathbf{x},\mathbf{x}^{\prime }}^{2}}\leqslant \liminf_{N\rightarrow \infty
}\Vert P_{\leq M}^{(k)}R^{(k)}B_{N,j,k+1}\gamma _{N}^{(k+1)}(t)\Vert
_{L_{T}^{1}L_{\mathbf{x},\mathbf{x}^{\prime }}^{2}}\leqslant C^{k}.
\end{equation*}%
Since the above holds uniformly in $M$, we can send $M\rightarrow \infty $
and, by the monotone convergence theorem, we obtain 
\begin{equation*}
\Vert R^{(k)}B_{j,k+1}\gamma ^{(k+1)}\Vert _{L_{T}^{1}L_{\mathbf{x},\mathbf{x%
}^{\prime }}^{2}}\leqslant C^{k}
\end{equation*}%
which is exactly the Klainerman-Machedon space-time bound %
\eqref{bound:KMConjecture}. This completes the proof of Theorem \ref%
{THM:MainTHM}, assuming Theorem \ref{THM:LocalizedMainTHM}.

The rest of this paper is devoted to proving Theorem \ref%
{THM:LocalizedMainTHM}. We are going to establish estimate (\ref%
{estimate:main target}) for a sufficiently small $T$ which depends on the
controlling constant in condition (\ref{condition:energy condition Tr}) and
is independent of $k$, $N$ and $M,$ then a bootstrap argument together with
condition (\ref{condition:energy condition Tr}) give estimate (\ref%
{estimate:main target}) for every finite time at the price of a larger
constant $C$. Before we start, alert readers should keep in mind that, we
will mostly use the following form of condition (\ref{condition:energy
condition Tr}): 
\begin{equation}
\left\Vert S^{(k)}\gamma _{N}^{(k)}\right\Vert _{L_{t}^{\infty }L_{\mathbf{%
x,x}^{\prime }}^{2}}\leqslant C_{0}^{k}.  \label{condition:energy condition}
\end{equation}%
where $S^{(k)}=\dprod\limits_{j=1}^{k}\left( \left\langle \nabla
_{x_{j}}\right\rangle \left\langle \nabla _{x_{j}^{\prime }}\right\rangle
\right) $, because we will be working in $L^{2}$. To see how (\ref%
{condition:energy condition}) follows from (\ref{condition:energy condition
Tr}), one simply notices that 
\begin{eqnarray*}
&&\int \left\vert \left\langle \nabla _{x}\right\rangle \left\langle \nabla
_{x^{\prime }}\right\rangle \int \phi \left( x,r\right) \overline{\phi
\left( x^{\prime },r\right) }dr\right\vert ^{2}dxdx^{\prime } \\
&=&\int \left\vert \int \left\langle \nabla _{x}\right\rangle \phi \left(
x,r\right) \overline{\left\langle \nabla _{x^{\prime }}\right\rangle \phi
\left( x^{\prime },r\right) }dr\right\vert ^{2}dxdx^{\prime } \\
&\leqslant &\int \left( \int \left\langle \nabla _{x}\right\rangle \phi
\left( x,r\right) \overline{\left\langle \nabla _{x}\right\rangle \phi
\left( x,r\right) }dr\right) \left( \int \left\langle \nabla _{x^{\prime
}}\right\rangle \phi \left( x^{\prime },r\right) \overline{\left\langle
\nabla _{x^{\prime }}\right\rangle \phi \left( x^{\prime },r\right) }%
dr\right) dxdx^{\prime } \\
&=&\left( \int \phi \left( x,r\right) \overline{\left( 1-\triangle
_{x}\right) \phi \left( x,r\right) }dxdr\right) ^{2}.
\end{eqnarray*}

We start the proof of Theorem \ref{THM:LocalizedMainTHM} by rewriting
hierarchy (\ref{hierarchy:BBGKY hierarchy in operator form}) as 
\begin{eqnarray}
\gamma _{N}^{(k)}(t_{k}) &=&U^{(k)}(t_{k})\gamma
_{N,0}^{(k)}+\int_{0}^{t_{k}}U^{(k)}(t_{k}-t_{k+1})V_{N}^{(k)}\gamma
_{N}^{(k)}(t_{k+1})dt_{k+1}  \label{equation:short duhamel of BBGKY} \\
&&+\frac{N-k}{N}\int_{0}^{t_{k}}U^{(k)}(t_{k}-t_{k+1})B_{N}^{(k+1)}\gamma
_{N}^{(k+1)}(t_{k+1})dt_{k+1}  \notag
\end{eqnarray}%
with the short-hand notation: 
\begin{eqnarray*}
U^{(k)} &=&e^{it\triangle _{\mathbf{x}_{k}}}e^{-it\triangle _{\mathbf{x}%
_{k}^{\prime }}}, \\
V_{N}^{(k)}\gamma _{N}^{(k)} &=&\frac{1}{N}\sum_{1\leqslant i<j\leqslant k}%
\left[ V_{N}(x_{i}-x_{j}),\gamma _{N}^{(k)}\right] \\
B_{N}^{(k+1)}\gamma _{N}^{(k+1)} &=&\sum_{j=1}^{k}B_{N,j,k+1}\gamma
_{N}^{(k+1)}.
\end{eqnarray*}%
We omit the $i$ in front of the potential term and the interaction term so
that we do not need to keep track of its exact power.

Writing out the $l_{c}$th Duhamel-Born series of $\gamma _{N}^{(k)}$ by
iterating hierarchy \eqref{equation:short duhamel of BBGKY} $l_{c}$ times%
\footnote{%
Here, $l_{c}$ stands for "level of coupling" or "length/depth of coupling".
When $l_{c}=0$, we have (\ref{equation:short duhamel of BBGKY}) back.}, we
have 
\begin{eqnarray*}
\gamma _{N}^{(k)}(t_{k})
&=&U^{(k)}(t_{k})\gamma _{N,0}^{(k)} \\
&&+\frac{N-k}{N}\int_{0}^{t_{k}}U^{(k)}(t_{k}-t_{k+1})B_{N}^{(k+1)}U^{(k+1)}(t_{k+1})\gamma
_{N,0}^{(k+1)}dt_{k+1} \\
&&+\int_{0}^{t_{k}}U^{(k)}(t_{k}-t_{k+1})V_{N}^{(k)}\gamma
_{N}^{(k)}(t_{k+1})dt_{k+1} \\
&&+\frac{N-k}{N}\int_{0}^{t_{k}}U^{(k)}(t_{k}-t_{k+1})B_{N}^{(k+1)} \\
&& \qquad \times \int_{0}^{t_{k+1}}U^{(k+1)}(t_{k+1}-t_{k+2}) V_{N}^{(k+1)}\gamma_{N}^{(k+1)}(t_{k+2}) dt_{k+2}dt_{k+1} \\
&&+\frac{N-k}{N}\frac{N-k-1}{N}
\int_{0}^{t_{k}}U^{(k)}(t_{k}-t_{k+1})B_{N}^{(k+1)}%
 \\
&& \qquad \times \int_{0}^{t_{k+1}} U^{(k+1)}(t_{k+1}-t_{k+2}) B_{N}^{(k+2)}\gamma
_{N}^{(k+2)}(t_{k+2})dt_{k+2}dt_{k+1} \\
&=&...
\end{eqnarray*}%
After $l_{c}$ iterations 
\begin{equation}
\gamma _{N}^{(k)}(t_{k})=\QTR{sc}{FP}^{(k,l_{c})}(t_{k})+\QTR{sc}{PP}%
^{(k,l_{c})}(t_{k})+\QTR{sc}{IP}^{(k,l_{c})}(t_{k})
\label{E:k-level-expansion}
\end{equation}%
where the \emph{free part} at coupling level $l_{c}$ is given by 
\begin{eqnarray*}
&&\QTR{sc}{FP}^{(k,l_{c})} \\
&=&U^{(k)}(t_{k})\gamma _{N,0}^{(k)}+ \\
&&\sum_{j=1}^{l_{c}}\left( \dprod_{l=0}^{j-1}\frac{N-k-l}{N}\right)
\int_{0}^{t_{k}}\cdots
\int_{0}^{t_{k+j-1}}U^{(k)}(t_{k}-t_{k+1})B_{N}^{(k+1)}\cdots \\
&&\times U^{(k+j-1)}(t_{k+j-1}-t_{k+j})B_{N}^{(k+j)}\left(
U^{(k+j)}(t_{k+j})\gamma _{N,0}^{(k+j)}\right) dt_{k+1}\cdots dt_{k+j},
\end{eqnarray*}%
the \emph{potential part} is given by 
\begin{eqnarray*}
&& \QTR{sc}{PP}^{(k,l_{c})} \\
&=&\int_{0}^{t_{k}}U^{(k)}(t_{k}-t_{k+1})V_{N}^{(k)}\gamma
_{N}^{(k)}(t_{k+1})dt_{k+1}+\sum_{j=1}^{l_{c}}\left( \dprod_{l=0}^{j-1}\frac{%
N-k-l}{N}\right)  \notag  \label{Term:PotentialPart} \\
&&\times \int_{0}^{t_{k}}\cdots
\int_{0}^{t_{k+j-1}}U^{(k)}(t_{k}-t_{k+1})B_{N}^{(k+1)}\cdots
U^{(k+j-1)}(t_{k+j-1}-t_{k+j})B_{N}^{(k+j)}  \notag \\
&&\times \left(
\int_{0}^{t_{k+j}}U^{(k+j)}(t_{k+j}-t_{k+j+1})V_{N}^{(k+j)}\gamma
_{N}^{(k+j)}(t_{k+j+1})dt_{k+j+1}\right) dt_{k+1}\cdots dt_{k+j},  \notag
\end{eqnarray*}
and the \emph{interaction part} is given by 
\begin{eqnarray*}
&&\QTR{sc}{IP}^{(k,l_{c})} \\
&=&\left( \dprod_{l=0}^{l_{c}}\frac{N-k-l}{N}%
\right) \int_{0}^{t_{k}}\cdots
\int_{0}^{t_{k+l_{c}}}U^{(k)}(t_{k}-t_{k+1})B_{N}^{(k+1)}\cdots  \\
&&\qquad \cdots U^{(k+l_{c})}(t_{k+l_{c}}-t_{k+l_{c}+1})
B_{N}^{(k+l_{c}+1)}\left( \gamma
_{N}^{(k+l_{c}+1)}(t_{k+l_{c}+1})\right) dt_{k+1}\cdots dt_{k+l_{c}+1}.
\end{eqnarray*}%
By \eqref{E:k-level-expansion}, to establish (\ref{estimate:main target}),
it suffices to prove that 
\begin{equation}
\left\Vert P_{\leq M}^{(k-1)}R^{(k-1)}B_{N,1,k}\QTR{sc}{FP}%
^{(k,l_{c})}\right\Vert _{L_{T}^{1}L_{\mathbf{x,x}^{\prime }}^{2}}\leqslant
C^{k-1}  \label{estimate:FreePartofBBGKY}
\end{equation}%
\begin{equation}
\left\Vert P_{\leq M}^{(k-1)}R^{(k-1)}B_{N,1,k}\QTR{sc}{PP}%
^{(k,l_{c})}\right\Vert _{L_{T}^{1}L_{\mathbf{x,x}^{\prime }}^{2}}\leqslant
C^{k-1}  \label{estimate:PotentialPartofBBGKY}
\end{equation}%
\begin{equation}
\left\Vert P_{\leq M}^{(k-1)}R^{(k-1)}B_{N,1,k}\QTR{sc}{IP}%
^{(k,l_{c})}\right\Vert _{L_{T}^{1}L_{\mathbf{x,x}^{\prime }}^{2}}\leqslant
C^{k-1}  \label{estimate:InteractionPartofBBGKY}
\end{equation}%
for all $k\geqslant 2$ and for some $C$ and a sufficiently small $T$
determined by the controlling constant in condition 
\eqref{condition:energy
condition} and independent of $k$, $N$ and $M.$ We observe that $B_{N}^{(j)}$
has $2j$ terms inside so that each summand of $\gamma _{N}^{(k)}(t_{k})$
contains factorially many terms $\left( \sim \frac{\left( k+l_{c}\right) !}{%
k!}\right) $. We use the Klainerman-Machedon board game to combine them and
hence reduce the number of terms that need to be treated. Define 
\begin{equation*}
J_{N}^{(k,j)}(\underline{t}%
_{k+j})(f^{(k+j)})=U^{(k)}(t_{k}-t_{k+1})B_{N}^{(k+1)}\cdots
U^{(k+j-1)}(t_{k+j-1}-t_{k+j})B_{N}^{(k+j)}f^{(k+j)},
\end{equation*}%
where $\underline{t}_{k+j}$ means $\left( t_{k+1},\ldots ,t_{k+j}\right) ,$
then the Klainerman-Machedon board game implies the lemma.

\begin{lemma}[Klainerman-Machedon board game]
\label{lemma:Klainerman-MachedonBoardGameForBBGKY}One can express 
\begin{equation*}
\int_{0}^{t_{k}}\cdots \int_{0}^{t_{k+j-1}}J_{N}^{(k,j)}(\underline{t}%
_{k+j})(f^{(k+j)})d\underline{t}_{k+j}
\end{equation*}%
as a sum of at most $4^{j-1}$ terms of the form 
\begin{equation*}
\int_{D}J_{N}^{(k,j)}(\underline{t}_{k+j},\mu _{m})(f^{(k+j)})d\underline{t}%
_{k+j},
\end{equation*}%
or in other words, 
\begin{equation*}
\int_{0}^{t_{k}}\cdots \int_{0}^{t_{k+j-1}}J_{N}^{(k,j)}(\underline{t}%
_{k+j})(f^{(k+j)})d\underline{t}_{k+j}=\sum_{m}\int_{D}J_{N}^{(k,j)}(%
\underline{t}_{k+j},\mu _{m})(f^{(k+j)})d\underline{t}_{k+j}.
\end{equation*}%
Here $D\subset \lbrack 0,t_{k}]^{j}$, $\mu _{m}$ are a set of maps from $%
\{k+1,\ldots ,k+j\}$ to $\{k,\ldots ,k+j-1\}$ satisfying $\mu _{m}(k+1)=k$
and $\mu _{m}(l)<l$ for all $l,$ and 
\begin{eqnarray*}
&&J_{N}^{(k,j)}(\underline{t}_{k+j},\mu _{m})(f^{(k+j)}) \\
&=&U^{(k)}(t_{k}-t_{k+1})B_{N,k,k+1}U^{(k+1)}(t_{k+1}-t_{k+2})B_{N,\mu
_{m}(k+2),k+2}\cdots \\
&&\cdots U^{(k+j-1)}(t_{k+j-1}-t_{k+j})B_{N,\mu _{m}(k+j),k+j}(f^{(k+j)}).
\end{eqnarray*}
\end{lemma}

\begin{proof}
Lemma \ref{lemma:Klainerman-MachedonBoardGameForBBGKY} follows the exact
same proof as \cite[Theorem 3.4]{KlainermanAndMachedon}, the
Klainerman-Machedon board game, if one replaces $B_{j,k+1}$ by $B_{N,j,k+1}$
and notices that $B_{N,j,k+1}$ still commutes with $e^{it\triangle
_{x_{i}}}e^{-it\triangle _{x_{i}^{\prime }}}$ whenever $i\neq j$. This
argument reduces the number of terms by combining them.
\end{proof}

In the rest of this paper, we establish estimate (\ref%
{estimate:PotentialPartofBBGKY}) only. The reason is the following. On the
one hand, the proof of estimate (\ref{estimate:PotentialPartofBBGKY}) is
exactly the place that relies on the restriction $\beta \in (0,2/3)$ in this
paper. On the other hand, X.C. has already proven estimates %
\eqref{estimate:FreePartofBBGKY} and \eqref{estimate:InteractionPartofBBGKY}
as estimates (6.3) and (6.5) in \cite{Chen3DDerivation} without using any
frequency localization. For completeness, we include a proof of estimates (%
\ref{estimate:FreePartofBBGKY}) and (\ref{estimate:InteractionPartofBBGKY})
in Appendix \ref{appendix:estimating free and interaction}. Before we delve
into the proof of estimate (\ref{estimate:PotentialPartofBBGKY}), we remark
that the proof of estimates (\ref{estimate:FreePartofBBGKY}) and (\ref%
{estimate:PotentialPartofBBGKY}) is independent of the coupling level $l_{c}$
and we will take the coupling level $l_{c}$ to be $\ln N$ for estimate (\ref%
{estimate:InteractionPartofBBGKY}).\footnote{%
The technique of taking $l_{c}=\ln N$ for estimate (\ref%
{estimate:InteractionPartofBBGKY}) was first observed by T.Chen and N.Pavlovi%
\'{c} \cite{TChenAndNPSpace-Time}.}

\section{Estimate of the potential part}

\label{Sec:estimating potential part}

In this section, we prove estimate (\ref{estimate:PotentialPartofBBGKY}). To
be specific, we establish the following theorem.

\begin{theorem}
\label{THM:beta up to 2/3}Under the assumptions of Theorem \ref{THM:MainTHM}%
, there exists a $C$ (independent of $k,l_{c},M_{k-1},N$) such that for each 
$M_{k-1}\geqslant 1$ there exists $N_{0}$ (depending on $M_{k-1}$) such that
for $N\geqslant N_{0}$, there holds 
\begin{equation*}
\left\Vert P_{\leqslant M_{k-1}}^{(k-1)}R^{(k-1)}B_{N,1,k}\QTR{sc}{PP}%
^{(k,l_{c})}\right\Vert _{L_{T}^{1}L_{\mathbf{x,x}^{\prime }}^{2}}\leqslant
C^{k-1}
\end{equation*}%
where $\QTR{sc}{PP}^{(k,l_{c})}$ is given by (\ref{Term:PotentialPart}).
\end{theorem}

In this section, we will employ the estimates stated and proved in Section %
\ref{S:Strichartz-X}. Due to the technicality of the proof of Theorem \ref%
{THM:beta up to 2/3} involving Littlewood-Paley theory, we prove a simpler $%
\beta \in (0, \frac{2}{5})$ version first to illustrate the basic steps in
establishing Theorem \ref{THM:beta up to 2/3}. We then prove Theorem \ref%
{THM:beta up to 2/3} in Section \ref{Sec:proof of potential estimate for
beta up to 2/3}.

\subsection{A simpler proof in the case $\protect\beta\in(0,\frac25)$}

\begin{theorem}
\label{Theorem:example:beta up to 2/5} For $\beta \in (0,\frac{2}{5}),$ we
have the estimate 
\begin{equation*}
\left\Vert R^{(k-1)}B_{N,1,k}\QTR{sc}{PP}^{(k,l_{c})}\right\Vert
_{L_{T}^{1}L_{\mathbf{x,x}^{\prime }}^{2}}\leqslant C^{k-1}
\end{equation*}%
for some $C$ and a sufficiently small $T$ determined by the controlling
constant in condition \eqref{condition:energy condition} and independent of $%
k,l_{c}$ and $N.$
\end{theorem}

\begin{proof}
The proof is divided into four steps. We will reproduce every step for
Theorem \ref{THM:beta up to 2/3} in Section \ref{Sec:proof of potential
estimate for beta up to 2/3}.

\medskip

\noindent \emph{Step I}. By Lemma \ref%
{lemma:Klainerman-MachedonBoardGameForBBGKY}, we know that 
\begin{eqnarray}
&&\text{\textsc{PP}}^{(k,l_{c})}  \label{Term:PotentialTermAfterBoardGame} \\
&=&\int_{0}^{t_{k}}U^{(k)}(t_{k}-t_{k+1})V_{N}^{(k)}\gamma
_{N}^{(k)}(t_{k+1})dt_{k+1}  \notag \\
&&+\sum_{j=1}^{l_{c}}\left( \dprod_{l=0}^{j-1}\frac{N-k-l}{N}\right) \left(
\sum_{m}\int_{D}J_{N}^{(k,j)}(\underline{t}_{k+j},\mu _{m})\left(
f^{(k+j)}\right) d\underline{t}_{k+j}\right)  \notag
\end{eqnarray}%
where%
\begin{equation}
f^{(k+j)}=\int_{0}^{t_{k+j}}U^{(k+j)}(t_{k+j}-t_{k+j+1})V_{N}^{(k+j)}\gamma
_{N}^{(k+j)}(t_{k+j+1})dt_{k+j+1},  \label{def:f for potential term}
\end{equation}%
and the sum $\sum_{m}$ has at most $4^{j-1}$ terms inside.

For the second term in (\ref{Term:PotentialTermAfterBoardGame}), we iterate
Lemma \ref{Lemma:KMOriginalEstimateWithVN} to prove the following estimate%
\footnote{%
This also helps in proving estimates (\ref{estimate:FreePartofBBGKY}) and (%
\ref{estimate:InteractionPartofBBGKY})--see Appendix \ref%
{appendix:estimating free and interaction}}:%
\begin{eqnarray}
&&\left\Vert R^{(k-1)}B_{N,1,k}\int_{D}J_{N}^{(k,j)}(\underline{t}_{k+j},\mu
_{m})\left( f^{(k+j)}\right) d\underline{t}_{k+j}\right\Vert _{L_{T}^{1}L_{%
\mathbf{x,x}^{\prime }}^{2}}  \label{Relation:IteratngCollapsingEstimate} \\
&\leqslant &(CT^{\frac{1}{2}})^{j}\left\Vert R^{(k+j-1)}B_{N,\mu
_{m}(k+j),k+j}f^{(k+j)}\right\Vert _{L_{T}^{1}L_{\mathbf{x,x}^{\prime
}}^{2}}.  \notag
\end{eqnarray}%
In fact,%
\begin{eqnarray*}
&&\left\Vert R^{(k-1)}B_{N,1,k}\int_{D}J_{N}^{(k,j)}(\underline{t}_{k+j},\mu
_{m})\left( f^{(k+j)}\right) d\underline{t}_{k+j}\right\Vert _{L_{T}^{1}L_{%
\mathbf{x,x}^{\prime }}^{2}} \\
&=&\int_{0}^{T}\left\Vert
\int_{D}R^{(k-1)}B_{N,1,k}U^{(k)}(t_{k}-t_{k+1})B_{N,k,k+1}\cdots
dt_{k+1}\ldots dt_{k+j}\right\Vert _{L_{\mathbf{x,x}^{\prime }}^{2}}dt_{k}.
\end{eqnarray*}%
By Minkowski,%
\begin{equation*}
\leqslant \int_{\left[ 0,T\right] ^{j+1}}\left\Vert
R^{(k-1)}B_{N,1,k}U^{(k)}(t_{k}-t_{k+1})B_{N,k,k+1}\cdots \right\Vert _{L_{%
\mathbf{x,x}^{\prime }}^{2}}dt_{k}dt_{k+1}...dt_{k+j}.
\end{equation*}%
Cauchy-Schwarz in $dt_{k}$,%
\begin{equation*}
\leqslant T^{\frac{1}{2}}\int_{\left[ 0,T\right] ^{j}}\left( \int \left\Vert
R^{(k-1)}B_{N,1,k}U^{(k)}(t_{k}-t_{k+1})B_{N,k,k+1}...\right\Vert _{L_{%
\mathbf{x,x}^{\prime }}^{2}}^{2}dt_{k}\right) ^{\frac{1}{2}%
}dt_{k+1}...dt_{k+j}.
\end{equation*}%
Use Lemma $\ref{Lemma:KMOriginalEstimateWithVN},$%
\begin{equation*}
\leqslant CT^{\frac{1}{2}}\int_{\left[ 0,T\right] ^{j}}\left\Vert
R^{(k)}B_{N,k,k+1}U^{(k+1)}(t_{k+1}-t_{k+2})...\right\Vert _{L_{\mathbf{x,x}%
^{\prime }}^{2}}dt_{k+1}...dt_{k+j}.
\end{equation*}%
Repeat the previous steps for $j-1$ time, we then reach relation (\ref%
{Relation:IteratngCollapsingEstimate}).

Applying relation (\ref{Relation:IteratngCollapsingEstimate}) to (\ref%
{Term:PotentialTermAfterBoardGame}), we have 
\begin{eqnarray*}
&&\left\Vert R^{(k-1)}B_{N,1,k}\QTR{sc}{PP}^{(k,l_{c})}\right\Vert
_{L_{T}^{1}L_{\mathbf{x,x}^{\prime }}^{2}} \\
&\leqslant &\left\Vert
R^{(k-1)}B_{N,1,k}\int_{0}^{t_{k}}U^{(k)}(t_{k}-t_{k+1})V_{N}^{(k)}\gamma
_{N}^{(k)}(t_{k+1})dt_{k+1}\right\Vert _{L_{T}^{1}L_{\mathbf{x,x}^{\prime
}}^{2}} \\
&&+\sum_{j=1}^{l_{c}}4^{j-1}(CT^{\frac{1}{2}})^{j}\left\Vert
R^{(k+j-1)}B_{N,\mu _{m}(k+j),k+j}\left( f^{(k+j)}\right) \right\Vert
_{L_{T}^{1}L_{\mathbf{x,x}^{\prime }}^{2}} \\
&\leqslant &\left\Vert
R^{(k-1)}B_{N,1,k}\int_{0}^{t_{k}}U^{(k)}(t_{k}-t_{k+1})V_{N}^{(k)}\gamma
_{N}^{(k)}(t_{k+1})dt_{k+1}\right\Vert _{L_{T}^{1}L_{\mathbf{x,x}^{\prime
}}^{2}} \\
&&+\sum_{j=1}^{l_{c}}(CT^{\frac{1}{2}})^{j}\left\Vert R^{(k+j-1)}B_{N,\mu
_{m}(k+j),k+j}\left( f^{(k+j)}\right) \right\Vert _{L_{T}^{1}L_{\mathbf{x,x}%
^{\prime }}^{2}}.
\end{eqnarray*}%
Inserting a smooth cut-off $\theta (t)$ with $\theta (t)=1$ for $t\in \left[
-T,T\right] $ and $\theta (t)=0$ for $t\in \left[ -2T,2T\right] ^{c}$ into
the above estimate, we get 
\begin{eqnarray*}
&&\left\Vert R^{(k-1)}B_{N,1,k}\QTR{sc}{PP}^{(k,l_{c})}\right\Vert
_{L_{T}^{1}L_{\mathbf{x,x}^{\prime }}^{2}} \\
&\leqslant &\left\Vert R^{(k-1)}B_{N,1,k}\theta
(t_{k})\int_{0}^{t_{k}}U^{(k)}(t_{k}-t_{k+1})\theta
(t_{k+1})V_{N}^{(k)}\gamma _{N}^{(k)}(t_{k+1})dt_{k+1}\right\Vert
_{L_{T}^{1}L_{\mathbf{x,x}^{\prime }}^{2}} \\
&&+\sum_{j=1}^{l_{c}}(CT^{\frac{1}{2}})^{j}\left\Vert R^{(k+j-1)}B_{N,\mu
_{m}(k+j),k+j}\theta (t_{k+j})\left( \tilde{f}^{(k+j)}\right) \right\Vert
_{L_{T}^{1}L_{\mathbf{x,x}^{\prime }}^{2}}
\end{eqnarray*}%
where%
\begin{equation}
\tilde{f}^{(k+j)}=\int_{0}^{t_{k+j}}U^{(k+j)}(t_{k+j}-t_{k+j+1})\left(
\theta (t_{k+j+1})V_{N}^{(k+j)}\gamma _{N}^{(k+j)}(t_{k+j+1})\right)
dt_{k+j+1}  \label{def:f tutle for potential term}
\end{equation}

\noindent \emph{Step II}. The $X_{b}$ space version of Lemma \ref%
{Lemma:KMOriginalEstimateWithVN}, Lemma \ref{Lemma:KMEstimateInWithX_b},
then turns the last step into 
\begin{eqnarray*}
&&\left\Vert R^{(k-1)}B_{N,1,k}\QTR{sc}{PP}^{(k,l_{c})}\right\Vert
_{L_{T}^{1}L_{\mathbf{x,x}^{\prime }}^{2}} \\
&\leqslant &CT^{\frac{1}{2}}\Vert \theta
(t_{k})\int_{0}^{t_{k}}U^{(k)}(t_{k}-t_{k+1})R^{(k)}\left( \theta
(t_{k+1})V_{N}^{(k)}\gamma _{N}^{(k)}(t_{k+1})\right) dt_{k+1}\Vert _{X_{%
\frac{1}{2}+}^{(k)}} \\
&&+C\sum_{j=1}^{l_{c}}(CT^{\frac{1}{2}})^{j+1}\Vert \theta (t_{k+j})R^{(k+j)}%
\tilde{f}^{(k+j)}\Vert _{X_{\frac{1}{2}+}^{(k+j)}}
\end{eqnarray*}

\noindent \emph{Step III}. Recall the definition of $\tilde{f}^{(k+j)},$ 
\begin{equation*}
\tilde{f}^{(k+j)}=\int_{0}^{t_{k+j}}U^{(k+j)}(t_{k+j}-t_{k+j+1})\left(
\theta (t_{k+j+1})V_{N}^{(k+j)}\gamma _{N}^{(k+j)}(t_{k+j+1})\right)
dt_{k+j+1}
\end{equation*}%
so%
\begin{eqnarray*}
&&R^{(k+j)}\tilde{f}^{(k+j)} \\
&=&\int_{0}^{t_{k+j}}U^{(k+j)}(t_{k+j}-t_{k+j+1})R^{(k+j)}\left( \theta
(t_{k+j+1})V_{N}^{(k+j)}\gamma _{N}^{(k+j)}(t_{k+j+1})\right) dt_{k+j+1}.
\end{eqnarray*}%
We then proceed with Lemma \ref{Lemma:b to b-1} to get 
\begin{eqnarray*}
&&\left\Vert R^{(k-1)}B_{N,1,k}\QTR{sc}{PP}^{(k,l_{c})}\right\Vert
_{L_{T}^{1}L_{\mathbf{x,x}^{\prime }}^{2}} \\
&\leqslant &CT^{\frac{1}{2}}\Vert R^{(k)}\left( \theta
(t_{k+1})V_{N}^{(k)}\gamma _{N}^{(k)}(t_{k+1})\right) \Vert _{X_{-\frac{1}{2}%
+}^{(k)}} \\
&&+C\sum_{j=1}^{l_{c}}(CT^{\frac{1}{2}})^{j+1}\Vert R^{(k+j)}\left( \theta
(t_{k+j+1})V_{N}^{(k+j)}\gamma _{N}^{(k+j)}(t_{k+j+1})\right) \Vert _{X_{-%
\frac{1}{2}+}^{(k+j)}}.
\end{eqnarray*}

\noindent \emph{Step IV}. Now we would like to utilize Lemma \ref%
{Lemma:TheStrichartzEstimate}. We first analyse a typical term to
demonstrated the effect of Lemma \ref{Lemma:TheStrichartzEstimate}. To be
specific, we have 
\begin{eqnarray*}
&&\Vert R^{(k)}\left( \theta (t_{k+1})V_{N}(x_{1}-x_{2})\gamma
_{N}^{(k)}(t_{k+1})\right) \Vert _{X_{-\frac{1}{2}+}^{(k)}} \\
&\leqslant &\frac{C}{N}\Vert V_{N}(x_{1}-x_{2})\theta (t_{k+1})R^{(k)}\gamma
_{N}^{(k)}(t_{k+1})\Vert _{X_{-\frac{1}{2}+}^{(k)}} \\
&&+\frac{C}{N}\Vert \left( V_{N}\right) ^{\prime }(x_{1}-x_{2})\theta
(t_{k+1})\left( \frac{R^{(k)}}{\left\vert \nabla _{x_{1}}\right\vert }%
\right) \gamma _{N}^{(k)}(t_{k+1})\Vert _{X_{-\frac{1}{2}+}^{(k)}} \\
&&+\frac{C}{N}\Vert \left( V_{N}\right) ^{\prime \prime }(x_{1}-x_{2})\theta
(t_{k+1})\left( \frac{R^{(k)}}{\left\vert \nabla _{x_{1}}\right\vert
\left\vert \nabla _{x_{2}}\right\vert }\right) \gamma
_{N}^{(k)}(t_{k+1})\Vert _{X_{-\frac{1}{2}+}^{(k)}}.
\end{eqnarray*}%
by Leibniz's rule, where%
\begin{equation*}
\frac{R^{(k)}}{\left\vert \nabla _{x_{1}}\right\vert }=\left(
\dprod_{j=2}^{k}\left\vert \nabla _{x_{j}}\right\vert \right) \left(
\dprod_{j=1}^{k}\left\vert \nabla _{x_{j}^{\prime }}\right\vert \right) .
\end{equation*}%
Utilize Lemma \ref{Lemma:TheStrichartzEstimate} to each summand of the
above, we have 
\begin{eqnarray*}
&\leqslant &\frac{C}{N}\left\Vert V_{N}\right\Vert _{L^{3+}}\Vert \theta
(t_{k+1})R^{(k)}\gamma _{N}^{(k)}\Vert _{L_{t_{k+1}}^{2}L_{x,x^{\prime
}}^{2}} \\
&&+\frac{C}{N}\left\Vert V_{N}^{\prime }\right\Vert _{L^{2+}}\Vert \theta
(t_{k+1})\left\langle \nabla _{x_{1}}\right\rangle ^{\frac{1}{2}}\left( 
\frac{R^{(k)}}{\left\vert \nabla _{x_{1}}\right\vert }\right) \gamma
_{N}^{(k)}\Vert _{L_{t_{k+1}}^{2}L_{x,x^{\prime }}^{2}} \\
&&+\frac{C}{N}\left\Vert V_{N}^{\prime \prime }\right\Vert _{L^{\frac{6}{5}%
+}}\Vert \theta (t_{k+1})\left\langle \nabla _{x_{1}}\right\rangle
\left\langle \nabla _{x_{2}}\right\rangle \left( \frac{R^{(k)}}{\left\vert
\nabla _{x_{1}}\right\vert \left\vert \nabla _{x_{2}}\right\vert }\right)
\gamma _{N}^{(k)}\Vert _{L_{t_{k+1}}^{2}L_{x,x^{\prime }}^{2}} \\
&\leqslant &C\Vert S^{(k)}\gamma _{N}^{(k)}\Vert _{L_{2T}^{2}L_{\mathbf{x},%
\mathbf{x}^{\prime }}^{2}},
\end{eqnarray*}%
i.e.%
\begin{equation*}
\Vert R^{(k)}\left( \theta (t_{k+1})V_{N}(x_{1}-x_{2})\gamma
_{N}^{(k)}(t_{k+1})\right) \Vert _{X_{-\frac{1}{2}+}^{(2)}}\leqslant C\Vert
S^{(k)}\gamma _{N}^{(k)}\Vert _{L_{2T}^{2}L_{\mathbf{x},\mathbf{x}^{\prime
}}^{2}},
\end{equation*}%
since $\left\Vert V_{N}/N\right\Vert _{L^{3+}}$ $\left\Vert V_{N}^{\prime
}/N\right\Vert _{L^{2+}}$, and $\left\Vert V_{N}^{\prime \prime
}/N\right\Vert _{L^{\frac{6}{5}+}}$ are uniformly bounded in $N$ for $\beta
\in (0,\frac{2}{5})$. In fact,%
\begin{eqnarray*}
\left\Vert V_{N}/N\right\Vert _{L^{3+}} &\leqslant &N^{2\beta -1}\left\Vert
V\right\Vert _{L^{3+}} \\
\left\Vert V_{N}^{\prime }/N\right\Vert _{L^{2+}} &\leqslant &N^{\frac{%
5\beta }{2}-1}\left\Vert V^{\prime }\right\Vert _{L^{2+}} \\
\left\Vert V_{N}^{\prime \prime }/N\right\Vert _{L^{\frac{6}{5}+}}
&\leqslant &N^{\frac{5\beta }{2}-1}\left\Vert V^{\prime \prime }\right\Vert
_{L^{\frac{6}{5}+}}
\end{eqnarray*}%
where by Sobolev, $V\in W^{2,\frac{6}{5}+}$ implies $V\in L^{\frac{6}{5}%
+}\cap L^{6+}$ and $V^{\prime }\in L^{2+}$.

Using the same idea for all the terms, we end up with 
\begin{eqnarray*}
&&\left\Vert R^{(k-1)}B_{N,1,k}\QTR{sc}{PP}^{(k,l_{c})}\right\Vert
_{L_{T}^{1}L_{\mathbf{x,x}^{\prime }}^{2}} \\
&\leqslant &CTk^{2}\Vert S^{(k)}\gamma _{N}^{(k)}\Vert _{L_{2T}^{\infty }L_{%
\mathbf{x},\mathbf{x}^{\prime }}^{2}}+CT^{\frac{1}{2}}\sum_{j=1}^{l_{c}}(CT^{%
\frac{1}{2}})^{j+1}\left( k+j\right) ^{2}\Vert S^{(k+j)}\gamma
_{N}^{(k+j)}\Vert _{L_{2T}^{\infty }L_{\mathbf{x},\mathbf{x}^{\prime }}^{2}}
\end{eqnarray*}%
because there are $k^{2}$ terms inside $V_{N}^{(k)}$. Plug in Condition (\ref%
{condition:energy condition}), 
\begin{eqnarray*}
&\leqslant &CTk^{2}C_{0}^{k}+CT^{\frac{1}{2}}\sum_{j=1}^{\infty }(CT^{\frac{1%
}{2}})^{j+1}\left( k+j\right) ^{2}C_{0}^{k+j} \\
&\leqslant &C_{0}^{k}\left( CTk^{2}+CT^{\frac{1}{2}}k^{2}\sum_{j=1}^{\infty
}(CT^{\frac{1}{2}})^{j+1}C_{0}^{j}+CT^{\frac{1}{2}}\sum_{j=1}^{\infty }(CT^{%
\frac{1}{2}})^{j+1}j^{2}C_{0}^{j}\right) .\text{ }
\end{eqnarray*}%
We can then choose a $T$ independent of $k,$ $l_{c}$ and $N$ such that the
two infinite series converge. We then have%
\begin{eqnarray*}
\left\Vert R^{(k-1)}B_{N,1,k}\QTR{sc}{PP}^{(k,l_{c})}\right\Vert
_{L_{T}^{1}L_{\mathbf{x,x}^{\prime }}^{2}} &\leqslant &C_{0}^{k}\left(
CTk^{2}+CT^{\frac{1}{2}}k^{2}+CT^{\frac{1}{2}}\right) \\
&\leqslant &C_{0}^{k}\left( CT2^{k}+CT^{\frac{1}{2}}2^{k}+CT^{\frac{1}{2}%
}\right) \\
&\leqslant &C^{k-1}
\end{eqnarray*}%
for some $C$ larger than $C_{0}$ because $k\geqslant 2$. This concludes the
proof of Theorem \ref{Theorem:example:beta up to 2/5}.
\end{proof}

\subsection{The case $\protect\beta \in (0,\frac{2}{3})$}

\label{Sec:proof of potential estimate for beta up to 2/3}

To make formulas shorter, let us write 
\begin{equation*}
R_{\leqslant M_{k}}^{(k)}=P_{\leqslant M_{k}}^{(k)}R^{(k)},
\end{equation*}%
since $P_{\leqslant M_{k}}^{(k)}$ and $R^{(k)}$ are usually bundled together.

\subsubsection{Step I}

By \eqref{Term:PotentialTermAfterBoardGame},%
\begin{eqnarray*}
&&\left\Vert R_{\leqslant M_{k-1}}^{(k-1)}B_{N,1,k}\QTR{sc}{PP}%
^{(k,l_{c})}\right\Vert _{L_{T}^{1}L_{\mathbf{x,x}^{\prime }}^{2}} \\
&\leqslant &\left\Vert R_{\leqslant
M_{k-1}}^{(k-1)}B_{N,1,k}\int_{0}^{t_{k}}U^{(k)}(t_{k}-t_{k+1})V_{N}^{(k)}%
\gamma _{N}^{(k)}(t_{k+1})dt_{k+1}\right\Vert _{L_{T}^{1}L_{\mathbf{x,x}%
^{\prime }}^{2}} \\
&&+\sum_{j=1}^{l_{c}}\sum_{m}\Big\Vert R_{\leqslant
M_{k-1}}^{(k-1)}B_{N,1,k}\int_{D}J_{N}^{(k,j)}(\underline{t}_{k+j},\mu
_{m})\left( f^{(k+j)}\right) d\underline{t}_{k+j}\Big\Vert_{L_{T}^{1}L_{%
\mathbf{x,x}^{\prime }}^{2}}
\end{eqnarray*}%
where $f^{(k+j)}$ is again given by (\ref{def:f for potential term}) and the
sum $\sum_{m}$ has at most $4^{j-1}$ terms inside. By Minkowski's integral
inequality,%
\begin{eqnarray*}
&&\left\Vert R_{\leqslant M_{k-1}}^{(k-1)}B_{N,1,k}\int_{D}J_{N}^{(k,j)}(%
\underline{t}_{k+j},\mu _{m})\left( f^{(k+j)}\right) d\underline{t}%
_{k+j}\right\Vert _{L_{T}^{1}L_{\mathbf{x,x}^{\prime }}^{2}} \\
&=&\int_{0}^{T}dt_{k}\left\Vert \int_{D}R_{\leqslant
M_{k-1}}^{(k-1)}B_{N,1,k}J_{N}^{(k,j)}(\underline{t}_{k+j},\mu _{m})\left(
f^{(k+j)}\right) d\underline{t}_{k+j}\right\Vert _{L_{\mathbf{x,x}^{\prime
}}^{2}}dt_{k} \\
&\leqslant &\int_{[0,T]^{j+1}}\left\Vert R_{\leqslant
M_{k-1}}^{(k-1)}B_{N,1,k}U^{(k)}(t_{k}-t_{k+1})B_{N,k,k+1}...\right\Vert
_{L_{\mathbf{x,x}^{\prime }}^{2}}dt_{k}d\underline{t}_{k+j}
\end{eqnarray*}
By Cauchy-Schwarz in the $t_{k}$ integration, 
\begin{equation*}
\leqslant T^{\frac{1}{2}}\int_{\left[ 0,T\right] ^{j}}\left( \int \left\Vert
R_{\leqslant
M_{k-1}}^{(k-1)}B_{N,1,k}U^{(k)}(t_{k}-t_{k+1})B_{N,k,k+1}...\right\Vert
_{L_{\mathbf{x,x}^{\prime }}^{2}}dt_{k}\right) ^{\frac{1}{2}}d\underline{t}%
_{k+j}
\end{equation*}%
By Lemma \ref{Lemma:LocalizedKM}, 
\begin{equation*}
\leqslant C_{\varepsilon }T^{\frac{1}{2}}\sum_{M_{k}\geqslant M_{k-1}}\left( 
\frac{M_{k-1}}{M_{k}}\right) ^{1-\varepsilon }\int_{\left[ 0,T\right]
^{j}}\left\Vert R_{\leqslant
M_{k}}^{(k)}B_{N,k,k+1}U^{(k+1)}(t_{k+1}-t_{k+2})\cdots \right\Vert _{L_{%
\mathbf{x},\mathbf{x}^{\prime }}^{2}}d\underline{t}_{k+j}
\end{equation*}%
Iterating the previous step $(j-1)$ times,%
\begin{eqnarray*}
&\leqslant &(C_{\varepsilon }T^{\frac{1}{2}})^{j}\sum_{M_{k+j-1}\geqslant
\cdots \geqslant M_{k}\geqslant M_{k-1}}\Big[\left( \frac{M_{k-1}}{M_{k}}%
\frac{M_{k}}{M_{k+1}}\cdots \frac{M_{k+j-2}}{M_{k+j-1}}\right)
^{1-\varepsilon } \\
&&\times \left\Vert R_{\leqslant M_{k+j-1}}^{(k+j-1)}B_{N,\mu
_{m}(k+j),k+j}\left( f^{(k+j)}\right) \right\Vert _{L_{T}^{1}L_{\mathbf{x,x}%
^{\prime }}^{2}}\Big] \\
&=&(C_{\varepsilon }T^{\frac{1}{2}})^{j}\sum_{M_{k+j-1}\geqslant \cdots
\geqslant M_{k}\geqslant M_{k-1}}\Big[\left( \frac{M_{k-1}}{M_{k+j-1}}%
\right) ^{1-\varepsilon } \\
&&\times \left\Vert R_{\leqslant M_{k+j-1}}^{(k+j-1)}B_{N,\mu
_{m}(k+j),k+j}\left( f^{(k+j)}\right) \right\Vert _{L_{T}^{1}L_{\mathbf{x,x}%
^{\prime }}^{2}}\Big]
\end{eqnarray*}%
where the sum is over all $M_{k},\ldots ,M_{k+j-1}$ dyadic such that $%
M_{k+j-1}\geqslant \cdots \geqslant M_{k}\geqslant M_{k-1}$.

Hence%
\begin{eqnarray*}
&&\left\Vert R_{\leqslant M_{k-1}}^{(k-1)}B_{N,1,k}\QTR{sc}{PP}%
^{(k,l_{c})}\right\Vert _{L_{T}^{1}L_{\mathbf{x,x}^{\prime }}^{2}} \\
&\leqslant &\left\Vert R_{\leqslant
M_{k-1}}^{(k-1)}B_{N,1,k}\int_{0}^{t_{k}}U^{(k)}(t_{k}-t_{k+1})V_{N}^{(k)}%
\gamma _{N}^{(k)}(t_{k+1})dt_{k+1}\right\Vert _{L_{T}^{1}L_{\mathbf{x,x}%
^{\prime }}^{2}} \\
&&+\sum_{j=1}^{l_{c}}\Bigg\{(C_{\varepsilon }T^{\frac{1}{2}%
})^{j}\sum_{M_{k+j-1}\geqslant \cdots \geqslant M_{k}\geqslant M_{k-1}}\Big[%
\frac{M_{k-1}^{1-\varepsilon }}{M_{k+j-1}^{1-\varepsilon }} \\
&&\times \left\Vert R_{\leqslant M_{k+j-1}}^{(k+j-1)}B_{N,\mu
_{m}(k+j),k+j}\left( f^{(k+j)}\right) \right\Vert _{L_{T}^{1}L_{\mathbf{x,x}%
^{\prime }}^{2}}\Big]\Bigg\}
\end{eqnarray*}

We then insert a smooth cut-off $\theta (t)$ with $\theta (t)=1$ for $t\in %
\left[ -T,T\right] $ and $\theta (t)=0$ for $t\in \left[ -2T,2T\right] ^{c}$
into the above estimate to get%
\begin{eqnarray*}
&&\left\Vert R_{\leqslant M_{k-1}}^{(k-1)}B_{N,1,k}\QTR{sc}{PP}%
^{(k,l_{c})}\right\Vert _{L_{T}^{1}L_{\mathbf{x,x}^{\prime }}^{2}} \\
&\leqslant &\left\Vert R_{\leqslant M_{k-1}}^{(k-1)}B_{N,1,k}\theta
(t_{k})\int_{0}^{t_{k}}U^{(k)}(t_{k}-t_{k+1})\theta
(t_{k+1})V_{N}^{(k)}\gamma _{N}^{(k)}(t_{k+1})dt_{k+1}\right\Vert
_{L_{T}^{1}L_{\mathbf{x,x}^{\prime }}^{2}} \\
&&+\sum_{j=1}^{l_{c}}\Bigg\{(C_{\varepsilon }T^{\frac{1}{2}%
})^{j}\sum_{M_{k+j-1}\geqslant \cdots \geqslant M_{k}\geqslant M_{k-1}}\Big[%
\frac{M_{k-1}^{1-\varepsilon }}{M_{k+j-1}^{1-\varepsilon }} \\
&&\times \left\Vert R_{\leqslant M_{k+j-1}}^{(k+j-1)}B_{N,\mu
_{m}(k+j),k+j}\left( \theta (t_{k+j})\tilde{f}^{(k+j)}\right) \right\Vert
_{L_{T}^{1}L_{\mathbf{x,x}^{\prime }}^{2}}\Big]\Bigg\},
\end{eqnarray*}%
where the sum is over all $M_{k},\ldots ,M_{k+j-1}$ dyadic such that $%
M_{k+j-1}\geqslant \cdots \geqslant M_{k}\geqslant M_{k-1},$ and $\tilde{f}%
^{(k+j)}$ is again defined via (\ref{def:f tutle for potential term}).

\subsubsection{Step II}

Using Lemma \ref{Lemma:LocalizedKMWithX_b}, the $X_{b}$ space version of
Lemma \ref{Lemma:LocalizedKM}, we turn Step I into%
\begin{eqnarray*}
&&\left\Vert R_{\leqslant M_{k-1}}^{(k-1)}B_{N,1,k}\QTR{sc}{PP}%
^{(k,l_{c})}\right\Vert _{L_{T}^{1}L_{\mathbf{x,x}^{\prime }}^{2}} \\
&\leqslant &C_{\varepsilon }T^{\frac{1}{2}}\sum_{M_{k}\geqslant M_{k-1}} \Big[
\frac{M_{k-1}^{1-\varepsilon }}{M_{k}^{1-\varepsilon }} \\
&& \qquad \times \left\Vert
\int_{0}^{t_{k}}U^{(k)}(t_{k}-t_{k+1})\left( R_{\leqslant M_{k}}^{(k)}\theta
(t_{k+1})V_{N}^{(k)}\gamma _{N}^{(k)}(t_{k+1})\right) dt_{k+1}\right\Vert
_{X_{\frac{1}{2}+}^{(k)}} \Big] \\
&&+\sum_{j=1}^{l_{c}}(C_{\varepsilon }T^{\frac{1}{2}})^{j+1}\sum_{M_{k+j}%
\geqslant M_{k+j-1}\geqslant \cdots \geqslant M_{k}\geqslant M_{k-1}} \Big[\frac{%
M_{k-1}^{1-\varepsilon }}{M_{k+j}^{1-\varepsilon }} \\
&& \qquad \times \left\Vert \theta
(t_{k+j})R_{\leqslant M_{k+j}}^{(k+j)}\left( \tilde{f}^{(k+j)}\right)
\right\Vert _{X_{\frac{1}{2}+}^{(k+j)}} \Big]
\end{eqnarray*}

\subsubsection{Step III}

Lemma \ref{Lemma:b to b-1} gives us 
\begin{equation*}
\left\Vert R_{\leqslant M_{k-1}}^{(k-1)}B_{N,1,k}\QTR{sc}{PP}%
^{(k,l_{c})}\right\Vert _{L_{T}^{1}L_{\mathbf{x,x}^{\prime }}^{2}}\leqslant
A+B
\end{equation*}%
where%
\begin{equation*}
A=C_{\varepsilon }T^{\frac{1}{2}}\sum_{M_{k}\geqslant M_{k-1}}\frac{%
M_{k-1}^{1-\varepsilon }}{M_{k}^{1-\varepsilon }}\left\Vert R_{\leqslant
M_{k}}^{(k)}\theta (t_{k+1})V_{N}^{(k)}\gamma _{N}^{(k)}(t_{k+1})\right\Vert
_{X_{-\frac{1}{2}+}^{(k)}}
\end{equation*}%
and%
\begin{eqnarray*}
B &=&\sum_{j=1}^{l_{c}}\Bigg\{(C_{\varepsilon }T^{\frac{1}{2}%
})^{j+1}\sum_{M_{k+j}\geqslant M_{k+j-1}\geqslant \cdots \geqslant
M_{k}\geqslant M_{k-1}}\Big[\frac{M_{k-1}^{1-\varepsilon }}{%
M_{k+j}^{1-\varepsilon }} \\
&&\times \left\Vert R_{\leqslant M_{k+j}}^{(k+j)}\theta
(t_{k+j+1})V_{N}^{(k+j)}\gamma _{N}^{(k+j)}(t_{k+j+1})\right\Vert _{X_{-%
\frac{1}{2}+}^{(k+j)}}\Big]\Bigg\}
\end{eqnarray*}

\subsubsection{Step IV}

We focus for a moment on $B$. First, we carry out the sum in $M_{k}\leqslant
\cdots \leqslant M_{k+j-1}$ with the help of Lemma \ref{L:iterates3}:%
\begin{eqnarray*}
B &=&\sum_{j=1}^{l_{c}}\Bigg\{(C_{\varepsilon }T^{\frac{1}{2}%
})^{j+1}\sum_{M_{k+j}\geqslant M_{k-1}}\Big[\frac{M_{k-1}^{1-\varepsilon }}{%
M_{k+j}^{1-\varepsilon }}\left( \frac{(\log _{2}\frac{M_{k+j}}{M_{k-1}}%
+j)^{j}}{j!}\right) \\
&&\times \left\Vert R_{\leqslant M_{k+j}}^{(k+j)}\left( \theta
(t_{k+j+1})V_{N}^{(k+j)}\gamma _{N}^{(k+j)}(t_{k+j+1})\right) \right\Vert
_{X_{-\frac{1}{2}+}^{(k+j)}}\Big]\Bigg\}
\end{eqnarray*}
We then take a $T^{j/4}$ from the front to apply Lemma \ref{L:iterates4} and
get to%
\begin{eqnarray*}
B &\lesssim &C_{\varepsilon }T^{\frac{1}{2}}\sum_{j=1}^{l_{c}}\Bigg\{%
(C_{\varepsilon }T^{\frac{1}{4}})^{j}\sum_{M_{k+j}\geqslant M_{k-1}}\Big[%
\frac{M_{k-1}^{1-2\varepsilon }}{M_{k+j}^{1-2\varepsilon }} \\
&&\times \left\Vert R_{\leqslant M_{k+j}}^{(k+j)}\left( \theta
(t_{k+j+1})V_{N}^{(k+j)}\gamma _{N}^{(k+j)}(t_{k+j+1})\right) \right\Vert
_{X_{-\frac{1}{2}+}^{(k+j)}}\Big]\Bigg\}
\end{eqnarray*}%
where the sum is over dyadic $M_{k+j}$ such that $M_{k+j}\geqslant M_{k-1}$.
Applying \eqref{E:StrCor3},%
\begin{eqnarray*}
B &\lesssim &C_{\varepsilon }T^{\frac{1}{2}}\sum_{j=1}^{l_{c}}\Bigg\{%
(C_{\varepsilon }T^{\frac{1}{4}})^{j}(k+j)^{2}\sum_{M_{k+j}\geqslant M_{k-1}}%
\Big[\frac{M_{k-1}^{1-2\varepsilon }}{M_{k+j}^{1-2\varepsilon }}\min
(M_{k+j}^{2},N^{2\beta })N^{\frac{1}{2}\beta -1} \\
&&\times \left\Vert \theta (t_{k+j+1})S^{(k+j)}\gamma
_{N}^{(k+j)}(t_{k+j+1})\right\Vert _{L_{t_{k+j+1}}^{2}L_{\mathbf{x},\mathbf{x%
}^{\prime }}^{2}}\Big]\Bigg\}
\end{eqnarray*}
Rearranging terms%
\begin{eqnarray*}
B &\lesssim &C_{\varepsilon }T^{\frac{1}{2}}\sum_{j=1}^{l_{c}}\Bigg\{%
(C_{\varepsilon }T^{\frac{1}{4}})^{j}(k+j)^{2}\left\Vert \theta
(t_{k+j+1})S^{(k+j)}\gamma _{N}^{(k+j)}(t_{k+j+1})\right\Vert
_{L_{t_{k+j+1}}^{2}L_{\mathbf{x},\mathbf{x}^{\prime }}^{2}} \\
&&\times M_{k-1}^{1-2\epsilon }N^{\frac{1}{2}\beta -1}\sum_{M_{k+j}\geqslant
M_{k-1}}(\cdots )\Bigg\}
\end{eqnarray*}
where 
\begin{equation*}
\sum_{M_{k+j}\geqslant M_{k-1}}(\cdots )=\sum_{M_{k+j}\geqslant M_{k-1}}\min
(M_{k+j}^{1+2\epsilon },M_{k+j}^{-1+2\epsilon }N^{2\beta })\,.
\end{equation*}%
We carry out the sum in $M_{k+j}$ by dividing into $M_{k+j}\leqslant
N^{\beta }$ (for which $\min (M_{k+j}^{1+2\epsilon },M_{k+j}^{-1+2\epsilon
}N^{2\beta })=M_{j}^{1+2\epsilon }$) and $M_{k+j}\geqslant N^{\beta }$ (for
which $\min (M_{k+j}^{1+2\epsilon },M_{k+j}^{-1+2\epsilon }N^{2\beta
})=M_{k+j}^{-1+2\epsilon }N^{2\beta }$). This yields%
\begin{eqnarray*}
\sum_{M_{k+j}\geqslant M_{k-1}}\min (M_{k+j}^{1+2\epsilon
},M_{k+j}^{-1+2\epsilon }N^{2\beta }) &\lesssim &\left( \sum_{N^{\beta
}\geqslant M_{k+j}\geqslant M_{k-1}}+\sum_{M_{k+j}\geqslant
M_{k-1},M_{k+j}\geqslant N^{\beta }}\right) \left( ...\right) \\
&\lesssim &\sum_{N^{\beta }\geqslant M_{k+j}\geqslant 1}M_{k+j}^{1+2\epsilon
}+\sum_{M_{k+j}\geqslant N^{\beta }}M_{k+j}^{-1+2\epsilon }N^{2\beta } \\
&\lesssim &N^{\beta +2\epsilon }.
\end{eqnarray*}
Hence%
\begin{eqnarray*}
B &\lesssim &C_{\varepsilon }T^{\frac{1}{2}}\sum_{j=1}^{l_{c}} \Big[ (C_{%
\varepsilon }T^{\frac{1}{4}})^{j}(k+j)^{2} \\
&& \quad \times \left\Vert \theta
(t_{k+j+1})S^{(k+j)}\gamma _{N}^{(k+j)}(t_{k+j+1})\right\Vert
_{L_{t_{k+j+1}}^{2}L_{\mathbf{x},\mathbf{x}^{\prime
}}^{2}}M_{k-1}^{1-2\epsilon }N^{\frac{3}{2}\beta -1+2\epsilon } \Big] \\
&\lesssim &M_{k-1}^{1-2\epsilon }N^{\frac{3}{2}\beta -1+2\epsilon
}C_{\varepsilon }T^{\frac{1}{2}}\sum_{j=1}^{l_{c}}(C_{\varepsilon }T^{\frac{1%
}{4}})^{j}(k+j)^{2}T^{\frac{1}{2}}\left\Vert S^{(k+j)}\gamma
_{N}^{(k+j)}\right\Vert _{L_{t}^{\infty }L_{\mathbf{x},\mathbf{x}^{\prime
}}^{2}}
\end{eqnarray*}%
Via Condition $(\ref{condition:energy condition})$, it becomes%
\begin{eqnarray*}
B &\lesssim &M_{k-1}^{1-2\epsilon }N^{\frac{3}{2}\beta -1+2\epsilon
}C_{\varepsilon }T\sum_{j=1}^{l_{c}}(C_{\varepsilon }T^{\frac{1}{4}%
})^{j}(k+j)^{2}C_{0}^{k+j} \\
&\lesssim &C_{0}^{k}M_{k-1}^{1-2\epsilon }N^{\frac{3}{2}\beta -1+2\epsilon
}C_{\varepsilon }T\left( k^{2}\sum_{j=1}^{\infty }(C_{\varepsilon }T^{\frac{1%
}{4}})^{j}C_{0}^{j}+\sum_{j=1}^{\infty }(C_{\varepsilon }T^{\frac{1}{4}%
})^{j}j^{2}C_{0}^{j}\right)
\end{eqnarray*}%
We can then choose a $T$ independent of $M_{k-1}$, $k,$ $l_{c}$ and $N$ such
that the two infinite series converge. We then have%
\begin{equation*}
B\lesssim C^{k-1}M_{k-1}^{1-2\epsilon }N^{\frac{3}{2}\beta -1+2\epsilon }
\end{equation*}%
for some $C$ larger than $C_{0}.$ Therefore, for $\beta <2/3,$ there is a $C$
independent of $M_{k-1}$, $k,$ $l_{c},$ and $N$ s.t. given a $M_{k-1}$,
there is\ $N_{0}(M_{k-1})$ which makes 
\begin{equation*}
B\leqslant C^{k-1}\text{, for all }N\geqslant N_{0}.
\end{equation*}%
This completes the treatment of term $B$ for $\beta <2/3$. Term $A$ is
treated similarly (without the need to appeal to Lemmas \ref{L:iterates3}, %
\ref{L:iterates4} below). Whence we have completed the proof of Theorem \ref%
{THM:beta up to 2/3} and thence Theorem \ref{THM:LocalizedMainTHM}.

\begin{lemma}
\label{L:iterates3} 
\begin{equation*}
\left( \sum_{M_{k-1}\leq M_{k}\leq \cdots \leq M_{k+j-1}\leq
M_{k+j}}1\right) \leq \frac{(\log _{2}\frac{M_{k+j}}{M_{k-1}}+j)^{j}}{j!},
\end{equation*}%
where the sum is in $M_{k}\leq \cdots \leq M_{k+j-1}$ over dyads, such that $%
M_{k-1}\leq M_{k}\leq \cdots \leq M_{k+j-1}\leq M_{k+j}$.
\end{lemma}

\begin{proof}
This is equivalent to 
\begin{equation*}
\left( \sum_{i_{k-1}\leq i_{k}\leq \cdots \leq i_{k+j-1}\leq
i_{k+j}}1\right) \leq \frac{(i_{k+j}-i_{k-1}+j)^{j}}{j!},
\end{equation*}%
where the sum is taken over integers $i_{k},\ldots ,i_{k+j-1}$ such that $%
i_{k-1}\leq i_{k}\leq \cdots \leq i_{k+j-1}\leq i_{k+j}$. We use the
estimate (for $p\geq 0$, $\ell \geq 0$) 
\begin{equation*}
\sum_{i=0}^{q}(i+\ell )^{p}\leq \frac{(q+\ell +1)^{p+1}}{p+1},
\end{equation*}%
which just follows by estimating the sum by an integral.

First, carry out the sum in $i_{k}$ from $i_{k-1}$ to $i_{k+1}$ to obtain 
\begin{equation*}
=\sum_{i_{k-1}\leq i_{k+1}\leq \cdots \leq i_{k+j-1}\leq i_{k+j}}\left(
\sum_{i_{k}=i_{k-1}}^{i_{k+1}}1\right) \leq \sum_{i_{k-1}\leq i_{k+1}\leq
\cdots \leq i_{k+j-1}\leq i_{k+j}}(i_{k+1}-i_{k-1}+1).
\end{equation*}%
Next, carry out the sum in $i_{k+1}$ from $i_{k-1}$ to $i_{k+2}$, 
\begin{align*}
& \leq \sum_{i_{k-1}\leq i_{k+2}\leq \cdots \leq i_{k+j-1}\leq
i_{k+j}}\left( \sum_{i_{k+1}=i_{k-1}}^{i_{k+2}}(i_{k+1}-i_{k-1}+1)\right) \\
& \leq \sum_{i_{k-1}\leq i_{k+2}\leq \cdots \leq i_{k+j-1}\leq
i_{k+j}}\left( \sum_{i_{k+1}=0}^{i_{k+2}-i_{k-1}}(i_{k+1}+1)\right) \\
& \leq \sum_{i_{k-1}\leq i_{k+2}\leq \cdots \leq i_{k+j-1}\leq i_{k+j}}\frac{%
(i_{k+2}-i_{k-1}+2)^{2}}{2}.
\end{align*}%
Continue in this manner for $j-2$ times to obtain the claimed bound.
\end{proof}

\begin{lemma}
\label{L:iterates4} For each $\alpha>0$ (possibly large) and each $%
\epsilon>0 $ (arbitrarily small), there exists $t>0$ (independent of $M$)
sufficiently small such that 
\begin{equation*}
\forall \; j \geq 1, \; \forall \; M \,, \quad \text{we have} \quad \frac{
t^j (\alpha \log M + j)^j}{j!} \leq M^\epsilon
\end{equation*}
\end{lemma}

\begin{proof}
We use the following fact: for each $\sigma >0$ (arbitrarily small) there
exists $t>0$ sufficiently small such that 
\begin{equation}
\forall \;x>0\,,\qquad t^{x}\left( \frac{1}{x}+1\right) ^{x}\leq e^{\sigma }
\label{E:iterates10}
\end{equation}
To apply this fact to prove the lemma, use Stirling's formula to obtain 
\begin{equation*}
\frac{t^{j}(\alpha \log M+j)^{j}}{j!}\leq (et)^{j}\left( \frac{\alpha \log
M+j}{j}\right) ^{j}
\end{equation*}
Define $x$ in terms of $j$ by the formula $j=\alpha (\log M)x$. Then 
\begin{equation*}
=\left[ (et)^{x}\left( \frac{1}{x}+1\right) ^{x}\right] ^{\alpha \log M}
\end{equation*}
Applying \eqref{E:iterates10}, 
\begin{equation*}
\leq e^{\sigma \alpha \log M}=M^{\sigma \alpha }
\end{equation*}
\end{proof}

\section{Collapsing and Strichartz estimates}

\label{S:Strichartz-X}

Define the norm 
\begin{equation*}
\Vert \alpha ^{(k)}\Vert _{X_{b}^{(k)}}=\left( \int \langle \tau +\left\vert 
\mathbf{\xi }_{k}\right\vert ^{2}-\left\vert \mathbf{\xi }_{k}^{\prime
}\right\vert ^{2}\rangle ^{2b}\left\vert \hat{\alpha}^{(k)}(\tau ,\mathbf{%
\xi }_{k},\mathbf{\xi }_{k}^{\prime })\right\vert ^{2}\,d\tau\,d\mathbf{\xi }%
_{k}\,d\mathbf{\xi }_{k}^{\prime }\right) ^{1/2}
\end{equation*}

We will use the case $b=\frac{1}{2}+$ of the following lemma.

\begin{lemma}
\label{Lemma:b to b-1}Let $\frac12<b<1$ and $\theta (t)$ be a smooth cutoff.
Then 
\begin{equation}  \label{E:X-1}
\left\Vert \theta (t)\int_{0}^{t}U^{(k)}(t-s)\beta ^{(k)}(s)\,ds\right\Vert
_{X_{b}^{(k)}}\lesssim \Vert \beta ^{(k)}\Vert _{X_{b-1}^{(k)}}
\end{equation}
\end{lemma}

\begin{proof}
The estimate reduces to the space-independent estimate 
\begin{equation}
\left\Vert \theta (t)\int_{0}^{t}h(t^{\prime })\,dt^{\prime }\right\Vert
_{H_{t}^{b}}\lesssim \Vert h\Vert _{H_{t}^{b-1}},\qquad \text{for }\tfrac{1}{%
2}<b\leq 1  \label{E:integral}
\end{equation}%
Indeed, taking $h(t)=h_{\boldsymbol{x}_{k}\boldsymbol{x}_{k}^{\prime }}(t)%
\overset{\mathrm{def}}{=}U^{(k)}(-t)\beta ^{(k)}(t,\boldsymbol{x}_{k},%
\boldsymbol{x}_{k}^{\prime })$, applying the estimate \eqref{E:integral} for
fixed $\mathbf{x}_{k},\mathbf{x}_{k}^{\prime }$, and then applying the $L_{%
\mathbf{x}_{k}\mathbf{x}_{k}^{\prime }}^{2}$ norm to both sides, yields %
\eqref{E:X-1}. Now we prove estimate \eqref{E:integral}. Let $P_{\leq 1}$
and $P_{\geq 1}$ denote Littlewood-Paley projections onto frequencies $|\tau
|\lesssim 1$ and $|\tau |\gtrsim 1$ respectively. Decompose $h=P_{\leq
1}h+P_{\geq 1}h$ and use that $\int_{0}^{t}P_{\geq 1}h(t^{\prime })=\frac{1}{%
2}\int (\func{sgn}(t-t^{\prime })+\func{sgn}(t^{\prime }))P_{\geq
1}h(t^{\prime })\,dt^{\prime }$ to obtain the decomposition 
\begin{equation*}
\theta (t)\int_{0}^{t}h(t^{\prime })\,dt^{\prime
}=H_{1}(t)+H_{2}(t)+H_{3}(t),
\end{equation*}%
where 
\begin{align*}
& H_{1}(t)=\theta (t)\int_{0}^{t}P_{\leq 1}h(t^{\prime })\,dt^{\prime } \\
& H_{2}(t)=\tfrac{1}{2}\theta (t)[\func{sgn}\ast P_{\geq 1}h](t)\,dt^{\prime
} \\
& H_{3}(t)=\tfrac{1}{2}\theta (t)\int_{-\infty }^{+\infty }\func{sgn}%
(t^{\prime })P_{\geq 1}h(t^{\prime })\,dt^{\prime }.
\end{align*}%
We begin by addressing term $H_{1}$. By Sobolev embedding (recall $\frac{1}{2%
}<b\leq 1$) and the $L^{p}\rightarrow L^{p}$ boundedness of the Hilbert
transform for $1<p<\infty $, 
\begin{equation*}
\Vert H_{1}\Vert _{H_{t}^{b}}\lesssim \Vert H_{1}\Vert _{L_{t}^{2}}+\Vert
\partial _{t}H_{1}\Vert _{L_{t}^{2/(3-2b)}}\,.
\end{equation*}%
Using that $\Vert P_{\leq 1}h\Vert _{L_{t}^{\infty }}\lesssim \Vert h\Vert
_{H_{t}^{b-1}}$, we thus conclude 
\begin{equation*}
\Vert H_{1}\Vert _{H_{t}^{b}}\lesssim (\Vert \theta \Vert _{L_{t}^{2}}+\Vert
\theta \Vert _{L_{t}^{2/(3-2b)}}+\Vert \theta ^{\prime }\Vert
_{L_{t}^{2/3-2b}})\Vert h\Vert _{H_{t}^{b-1}}\,.
\end{equation*}

Next we address the term $H_2$. By the fractional Leibniz rule, 
\begin{equation*}
\|H_2\|_{H_t^b} \lesssim \|\langle D_t \rangle^b \theta \|_{L_t^2}\| \func{%
sgn} * P_{\geq 1}h \|_{L_t^\infty} + \|\theta\|_{L_t^\infty} \| \langle D_t
\rangle^b (\func{sgn}*P_{\geq 1}h) \|_{L_t^2} \,.
\end{equation*}
However, 
\begin{equation*}
\| \func{sgn} * P_{\geq 1}h \|_{L_t^\infty} \lesssim \| \langle \tau
\rangle^{-1} \hat h(\tau) \|_{L_\tau^1} \lesssim \| h \|_{H_t^{b-1}} \,.
\end{equation*}
On the other hand, 
\begin{equation*}
\| \langle D_t \rangle^b \func{sgn} * P_{\geq 1}h \|_{L_t^2} \lesssim \|
\langle \tau\rangle^b \langle \tau \rangle^{-1} \hat h(\tau) \|_{L_\tau^2}
\lesssim \|h\|_{H_t^{b-1}}\,.
\end{equation*}
Consequently, 
\begin{equation*}
\| H_2 \|_{H_t^b} \lesssim ( \| \langle D_t\rangle^b \theta \|_{L_t^2} +
\|\theta\|_{L_t^\infty}) \|h\|_{H_t^{b-1}} \,.
\end{equation*}
For term $H_3$, we have 
\begin{equation*}
\|H_3 \|_{H_t^b} \lesssim \|\theta \|_{H_t^b} \left\|
\int_{-\infty}^{+\infty} \func{sgn}(t^{\prime }) P_{\geq 1}h(t^{\prime }) \,
dt^{\prime }\right\|_{L_t^\infty} \,.
\end{equation*}
However, the second term is handled via Parseval's identity 
\begin{equation*}
\int_{t^{\prime }} \func{sgn}(t^{\prime }) P_{\geq 1}h(t^{\prime })
\,dt^{\prime }= \int_{|\tau|\geq 1} \tau^{-1} \hat h(\tau) \, d\tau \,,
\end{equation*}
from which the appropriate bounds follow again by Cauchy-Schwarz. Collecting
our estimates for $H_1$, $H_2$, and $H_3$, we have 
\begin{equation*}
\left\| \theta(t) \int_0^t h(t^{\prime }) \, dt^{\prime }\right\|_{H_t^b}
\lesssim C_\theta \|h\|_{H_t^{b-1}},
\end{equation*}
where 
\begin{equation*}
C_\theta = \|\theta\|_{L_t^2} + \|\theta^{\prime }\|_{L_t^{2/(3-2b)}} + \|
\langle D_t \rangle^b \theta\|_{L_t^2}+ \|\theta\|_{L_t^{2/(3-2b)}} +
\|\theta \|_{L_t^\infty}
\end{equation*}
\end{proof}

\subsection{Various forms of collapsing estimates}

\begin{lemma}
\label{Lemma:KMOriginalEstimateWithVN} There is a $C$ independent of $j,k$,
and $N$ such that, (for $f^{(k+1)}(\mathbf{x}_{k+1}, \mathbf{x}_{k+1})$
independent of $t$) 
\begin{equation*}
\left\Vert R^{(k)}B_{N,j,k+1}U^{(k+1)}(t)f^{(k+1)}\right\Vert _{L_{t}^{2}L_{ 
\mathbf{x},\mathbf{x}^{\prime }}^{2}}\leqslant C\left\Vert V\right\Vert
_{L^{1}}\left\Vert R^{(k+1)}f^{(k+1)}\right\Vert _{L_{\mathbf{x},\mathbf{x}
^{\prime }}^{2}}.
\end{equation*}
\end{lemma}

\begin{proof}
One can find this estimate as estimate (A.18) in \cite{TChenAndNPSpace-Time}
or a special case of Theorem 7 of \cite{Chen3DDerivation}. For more
estimates of this type, see \cite%
{Kirpatrick,GM,ChenDie,ChenAnisotropic,Beckner,Sohinger}.
\end{proof}

We have the following consequence of Lemma \ref%
{Lemma:KMOriginalEstimateWithVN}.

\begin{lemma}
\label{Lemma:KMEstimateInWithX_b}There is a $C$ independent of $j,k$, and $N$
such that (for $\alpha^{(k+1)}(t, \mathbf{x}_{k+1}, \mathbf{x}_{k+1})$
dependent on $t$) 
\begin{equation*}
\Vert R^{(k)}B_{N,j,k+1}\alpha ^{(k+1)}\Vert _{L_{t}^{2}L_{\mathbf{x}, 
\mathbf{x}^{\prime }}^{2}}\leqslant C\Vert R^{(k+1)}\alpha ^{(k+1)}\Vert
_{X_{\frac{1}{2}+}^{(k+1)}}
\end{equation*}
\end{lemma}

\begin{proof}
Let 
\begin{equation*}
f^{(k+1)}_\tau (\mathbf{x}_{k+1}, \mathbf{x}_{k+1}^{\prime }) = \mathcal{F}%
_{t\mapsto \tau} ( U^{(k+1)}(-t) \alpha^{(k+1)}(t,\mathbf{x}_{k+1}, \mathbf{x%
}_{k+1}^{\prime }))
\end{equation*}
where $\mathcal{F}_{t\mapsto \tau}$ denotes the Fourier transform in $%
t\mapsto \tau$. Then 
\begin{equation*}
\alpha^{(k+1)}(t,\mathbf{x}_{k+1},\mathbf{x}_{k+1}^{\prime }) = \int_\tau
e^{it\tau} U^{(k+1)}(t) f^{(k+1)}(\mathbf{x}_{k+1},\mathbf{x}_{k+1}^{\prime
}) \, d\tau
\end{equation*}
By Minkowski's inequality 
\begin{equation*}
\Vert R^{(k)}B_{N,j,k+1}\alpha ^{(k+1)}\Vert _{L_{t}^{2}L_{\mathbf{x}, 
\mathbf{x}^{\prime }}^{2}} \leq \int_\tau \Vert R^{(k)}B_{N,j,k+1}
U^{(k+1)}(t)f^{(k+1)} \Vert _{L_{t}^{2}L_{\mathbf{x}, \mathbf{x}^{\prime
}}^{2}} \, d\tau
\end{equation*}
By Lemma \ref{Lemma:KMOriginalEstimateWithVN}, 
\begin{equation*}
\leq \int_\tau \| R^{(k+1)}f^{(k+1)} \|_{L^2_{\mathbf{x},\mathbf{x}^{\prime
}}} \, d\tau
\end{equation*}
For any $b>\frac12$, we write $1=\langle \tau \rangle^{-b} \langle \tau
\rangle^b$ and apply Cauchy-Schwarz in $\tau$ to obtain 
\begin{equation*}
\leq \| \langle \tau \rangle^{b} R^{(k+1)}f^{(k+1)} \|_{L^2_{\tau, \mathbf{x}%
,\mathbf{x}^{\prime }}} = \| R^{(k+1)} \alpha^{(k+1)} \|_{X_b^{(k+1)}}
\end{equation*}
\end{proof}

\begin{lemma}
\label{Lemma:LocalizedKM} For each $\varepsilon >0$, there is a $%
C_{\varepsilon }$ independent of $M_{k},j,k$, and $N$ such that 
\begin{equation*}
\begin{aligned} \hspace{0.3in}&\hspace{-0.3in} \Vert R^{(k)}P_{\leqslant
M_{k}}^{(k)}B_{N,j,k+1}U^{(k+1)}(t)f^{(k+1)}\Vert
_{L_{t}^{2}L_{\mathbf{x},\mathbf{x}^{\prime }}^{2}} \\ &\leqslant
C_{\varepsilon }\left\Vert V\right\Vert _{L^{1}}\sum_{M_{k+1}\geqslant
M_{k}}\left( \frac{M_{k}}{M_{k+1}}\right) ^{1-\varepsilon }\left\Vert
R^{(k+1)}P_{\leqslant M_{k+1}}^{(k+1)}f^{(k+1)}\right\Vert
_{L_{\mathbf{x},\mathbf{x}^{\prime }}^{2}} \end{aligned}
\end{equation*}%
where the sum on the right is in $M_{k+1}$, over dyads such that $%
M_{k+1}\geqslant M_{k}$. In particular, if we drop off the projection $%
P_{\leqslant M_{k+1}}^{(k+1)}$ on the right hand side of the above estimate,
carry out the sum and let $M_{k}\rightarrow \infty ,$ we have Lemma \ref%
{Lemma:KMOriginalEstimateWithVN} back. This merely gives a fine structure of
Lemma \ref{Lemma:KMOriginalEstimateWithVN}, but not an alternative proof of
Lemma \ref{Lemma:KMOriginalEstimateWithVN}.
\end{lemma}

\begin{proof}
It suffices to take $k=1$ and prove 
\begin{eqnarray}
&&\Vert R^{(1)}P_{\leq
M_{1}}^{(1)}B_{N,1,2}(R^{(2)})^{-1}U^{(2)}(t)f^{(2)}\Vert _{L_{t}^{2}L_{%
\mathbf{x}_{1}\mathbf{x}_{1}^{\prime }}^{2}}
\label{estimate:the target for freq KM} \\
&\leq &C_{\varepsilon }\Vert V\Vert _{L^{1}}\sum_{M_{2}\geqslant
M_{1}}\left( \frac{M_{1}}{M_{2}}\right) ^{1-\varepsilon }\Vert P_{\leq
M_{2}}^{(2)}f^{(2)}\Vert _{L_{\mathbf{x}_{2}\mathbf{x}_{2}^{\prime }}^{2}} 
\notag
\end{eqnarray}%
where the sum is over dyadic $M_{2}$ such that $M_{2}\geqslant M_{1}$. For
convenience, we take only \textquotedblleft half\textquotedblright\ of the
operator $B_{N,1,2}$: For $\alpha ^{(2)}(t,x_{1},x_{2},x_{1}^{\prime
},x_{2}^{\prime })$, define 
\begin{equation*}
(\tilde{B}_{N,1,2}\alpha ^{(2)})(t,x_{1},x_{1}^{\prime })\overset{\mathrm{def%
}}{=}\int_{x_{2}}V_{N}(x_{1}-x_{2})\alpha ^{(2)}(t,x_{1},x_{2},x_{1}^{\prime
},x_{2})\,dx_{2}
\end{equation*}%
Note that 
\begin{align*}
\hspace{0.3in}& \hspace{-0.3in}\left( R^{(1)}P_{\leq M_{1}}^{(1)}\tilde{B}%
_{N,1,2}(R^{(2)})^{-1}U^{(2)}(t)f^{(2)}\right) \widehat{\;}(\tau ,\xi
_{1},\xi _{1}^{\prime }) \\
& =\iint_{\xi _{2},\xi _{2}^{\prime }}\chi _{\leqslant M_{1}}^{(1)}\delta
(\cdots )\frac{\widehat{V_{N}}(\xi _{2}+\xi _{2}^{\prime })|\xi _{1}|}{|\xi
_{1}-\xi _{2}-\xi _{2}^{\prime }||\xi _{2}||\xi _{2}^{\prime }|}\widehat{%
f^{(2)}}(\xi _{1}-\xi _{2}-\xi _{2}^{\prime },\xi _{2},\xi _{1}^{\prime
},\xi _{2}^{\prime })\,d\xi _{2}\,d\xi _{2}^{\prime }
\end{align*}%
where $\chi $ represents the Littlewood-Paley multiplier on the Fourier side
and 
\begin{equation*}
\delta (\cdots )=\delta (\tau +\left\vert \xi _{1}-\xi _{2}-\xi _{2}^{\prime
}\right\vert ^{2}+|\xi _{2}|^{2}-\left\vert \xi _{1}^{\prime }\right\vert
^{2}-\left\vert \xi _{2}^{\prime }\right\vert ^{2})
\end{equation*}%
Divide this integration into two pieces: 
\begin{equation*}
=\iint_{|\xi _{2}|\leq |\xi _{2}^{\prime }|}(\cdots )\,d\xi _{2}\,d\xi
_{2}^{\prime }+\iint_{|\xi _{2}^{\prime }|\leq |\xi _{2}|}(\cdots )\,d\xi
_{2}\,d\xi _{2}^{\prime }
\end{equation*}%
In the first term, decompose the $\xi _{2}^{\prime }$ integration into
dyadic intervals, and in the second term, decompose the $\xi _{2}$
integration into dyadic intervals: 
\begin{eqnarray*}
&=&A+B \\
&=&\left( \sum_{M_{2}\geq M_{1}}\iint_{|\xi _{2}|\leq |\xi _{2}^{\prime
}|}\chi _{M_{2}}^{2^{\prime }}(\cdots )\,d\xi _{2}\,d\xi _{2}^{\prime
}+\sum_{M_{2}\geq M_{1}}\iint_{|\xi _{2}^{\prime }|\leq |\xi _{2}|}\chi
_{M_{2}}^{2}(\cdots )\,d\xi _{2}\,d\xi _{2}^{\prime }\right) \\
&&+\left( \iint_{|\xi _{2}|\leq |\xi _{2}^{\prime }|}\chi _{\leq
M_{1}}^{2^{\prime }}(\cdots )\,d\xi _{2}\,d\xi _{2}^{\prime }+\iint_{|\xi
_{2}|\leq |\xi _{2}^{\prime }|}\chi _{\leq M_{1}}^{2}(\cdots )\,d\xi
_{2}\,d\xi _{2}^{\prime }\right) .
\end{eqnarray*}%
The $A$ term is the one that needs elaboration. For $B$, we have%
\begin{equation*}
B=\iint_{|\xi _{2}|\leq |\xi _{2}^{\prime }|}\chi _{\leqslant
M_{1}}^{(1)}\chi _{\leqslant M_{1}}^{2}\chi _{\leq M_{1}}^{2^{\prime
}}(\cdots )\,d\xi _{2}\,d\xi _{2}^{\prime } +\iint_{|\xi _{2}|\leq |\xi
_{2}^{\prime }|}\chi _{\leqslant M_{1}}^{(1)}\chi _{\leq M_{1}}^{2}\chi
_{\leqslant M_{1}}^{2^{\prime }}(\cdots )\,d\xi _{2}\,d\xi _{2}^{\prime }
\end{equation*}%
and thus, by Lemma \ref{Lemma:KMOriginalEstimateWithVN}, we reach 
\begin{eqnarray*}
\left\Vert B\right\Vert _{L_{\tau }^{2}L_{\xi _{1}\xi _{1}^{\prime }}^{2}}
&\leqslant &C\Vert V\Vert _{L^{1}}\Vert P_{\leq M_{1}}^{(2)}f^{(2)}\Vert
_{L_{\mathbf{x}_{2}\mathbf{x}_{2}^{\prime }}^{2}}
\end{eqnarray*}%
which is part of the right hand side of estimate (\ref{estimate:the target
for freq KM}).

We are now left with the estimate of $A.$ Observe that, in the first
integration in $A$, we can insert for free the projection $\chi _{\leq
3M_{2}}^{1}\chi _{\leq M_{1}}^{1^{\prime }}\chi _{\leq M_{2}}^{2}$ onto $%
\hat f^{(2)}$ and in the second integration, we can insert $\chi _{\leq
3M_{2}}^{1}\chi _{\leq M_{1}}^{1^{\prime }}\chi _{\leq M_{2}}^{2^{\prime }}$
onto $\hat f^{(2)}$. 
\begin{equation*}
A=\begin{aligned}[t] &\sum_{M_2 \geq M_1} \iint_{|\xi_2|\leq |\xi_2'|}
\chi_{\leq 3M_2}(\xi_1-\xi_2-\xi_2') \chi^{1'}_{\leq M_1} \chi^{2}_{\leq
M_2}\chi^{2'}_{M_2}(\cdots) \,d\xi_2\, d\xi_2' \\ &+ \sum_{M_2 \geq M_1}
\iint_{|\xi_2'|\leq |\xi_2|} \chi_{\leq 3M_2}(\xi_1-\xi_2-\xi_2')
\chi^{1'}_{\leq M_1} \chi^{2'}_{\leq M_2} \chi^2_{M_2} (\cdots) \,d\xi_2\,
d\xi_2' \end{aligned}
\end{equation*}%
Then for each piece, we proceed as in Klainerman-Machedon \cite%
{KlainermanAndMachedon}, performing Cauchy-Schwarz with respect to measures
supported on hypersurfaces and applying the $L_{\tau \xi _{1}\xi
_{1}^{\prime }}^{2}$ norm to both sides of the resulting inequality.%
\footnote{%
Notice that $\left\Vert \widehat{V_{N}}\right\Vert _{L^{\infty }}\leqslant
\left\Vert V_{N}\right\Vert _{L^{1}}=\left\Vert V\right\Vert _{L^{1}}$ i.e. $%
\widehat{V_{N}}$ is a dummy factor.} In this manner, it suffices to prove
the following estimates, uniform in $\tau ^{\prime }=\tau -\left\vert \xi
_{1}^{\prime }\right\vert ^{2}$: 
\begin{equation}
\iint_{\substack{ |\xi _{2}^{\prime }|\sim M_{2},  \\ |\xi _{2}|\leq M_{2}}}%
\delta (\cdots )\frac{|\xi _{1}|^{2}}{\left\vert \xi _{1}-\xi _{2}-\xi
_{2}^{\prime }\right\vert ^{2}\left\vert \xi _{2}\right\vert ^{2}\left\vert
\xi _{2}^{\prime }\right\vert ^{2}}\,d\xi _{2}\,d\xi _{2}^{\prime }\leq
C_{\varepsilon }\left( \frac{M_{1}}{M_{2}}\right) ^{2(1-\varepsilon )},
\label{E:KMloc1}
\end{equation}%
(recall that $|\xi _{1}|\lesssim M_{1}\ll M_{2}$) and also 
\begin{equation}
\iint_{\substack{ |\xi _{2}|\sim M_{2},  \\ |\xi _{2}^{\prime }|\leq M_{2}}}%
\delta (\cdots )\frac{|\xi _{1}|^{2}}{\left\vert \xi _{1}-\xi _{2}-\xi
_{2}^{\prime }\right\vert ^{2}\left\vert \xi _{2}\right\vert ^{2}\left\vert
\xi _{2}^{\prime }\right\vert ^{2}}\,d\xi _{2}\,d\xi _{2}^{\prime }\leq
C_{\varepsilon }\left( \frac{M_{1}}{M_{2}}\right) ^{2(1-\varepsilon )}.
\label{E:KMloc2}
\end{equation}%
In both \eqref{E:KMloc1} and \eqref{E:KMloc2}, 
\begin{equation*}
\delta (\cdots )=\delta (\tau ^{\prime }+\left\vert \xi _{1}-\xi _{2}-\xi
_{2}^{\prime }\right\vert ^{2}+\left\vert \xi _{2}\right\vert
^{2}-\left\vert \xi _{2}^{\prime }\right\vert ^{2}).
\end{equation*}%
By rescaling $\xi _{2}\mapsto M_{2}\xi _{2}$ and $\xi _{2}^{\prime }\mapsto
M_{2}\xi _{2}^{\prime }$, \eqref{E:KMloc1} and \eqref{E:KMloc2} reduce to,
respectively, the following.  For $|\xi _{1}|\ll 1$,
\begin{equation}
 I(\tau ^{\prime },\xi _{1})\overset{\mathrm{%
def}}{=}\iint_{\substack{ |\xi _{2}^{\prime }|\sim 1,  \\ |\xi _{2}|\leq 2}}%
\delta (\cdots )\frac{|\xi _{1}|^{2}}{\left\vert \xi _{1}-\xi _{2}-\xi
_{2}^{\prime }\right\vert ^{2}\left\vert \xi _{2}\right\vert ^{2}\left\vert
\xi _{2}^{\prime }\right\vert ^{2}}\,d\xi _{2}\,d\xi _{2}^{\prime }\leq
C_{\varepsilon }|\xi _{1}|^{2(1-\varepsilon )},  \label{E:KMloc3}
\end{equation}%
\begin{equation}
 I^{\prime }(\tau ^{\prime },\xi _{1})%
\overset{\mathrm{def}}{=}\iint_{\substack{ |\xi _{2}|\sim 1,  \\ |\xi
_{2}^{\prime }|\leq 2}}\delta (\cdots )\frac{|\xi _{1}|^{2}}{\left\vert \xi
_{1}-\xi _{2}-\xi _{2}^{\prime }\right\vert ^{2}\left\vert \xi
_{2}\right\vert ^{2}\left\vert \xi _{2}^{\prime }\right\vert ^{2}}\,d\xi
_{2}\,d\xi _{2}^{\prime }\leq C_{\varepsilon }|\xi _{1}|^{2(1-\varepsilon )}.
\label{E:KMloc4}
\end{equation}%
To be precise, the $\xi _{1}$ in estimates \eqref{E:KMloc3} and %
\eqref{E:KMloc4} is $\xi _{1}/M_{2}$ in estimates \eqref{E:KMloc1} and %
\eqref{E:KMloc2}. We shall obtain the upper bound $|\xi _{1}|^{2}\log |\xi
_{1}|^{-1}$ for both \eqref{E:KMloc3}, \eqref{E:KMloc4}.

First, we prove \eqref{E:KMloc4}. Begin by carrying out the $\xi
_{2}^{\prime }$ integral to obtain 
\begin{equation*}
I^{\prime }(\tau ^{\prime },\xi _{1})=\frac{1}{2}|\xi _{1}|^{2}\int_{\frac{1%
}{2}\leq |\xi _{2}|\leq 2}\frac{H^{\prime }(\tau ^{\prime },\xi _{1},\xi
_{2})}{|\xi _{1}-\xi _{2}||\xi _{2}|^{2}}\;d\xi _{2}
\end{equation*}%
where $H^{\prime }(\tau ^{\prime },\xi _{1},\xi _{2})$ is defined as
follows. Let $P^{\prime }$ be the truncated plane defined by 
\begin{equation*}
P^{\prime }(\tau ^{\prime },\xi _{1},\xi _{2})=\left\{ \xi _{2}^{\prime }\in 
\mathbb{R}^{3}\;|\;(\xi _{2}^{\prime }-\lambda \omega )\cdot \omega
=0\,,\;|\xi _{2}^{\prime }|\leq 2\right\}
\end{equation*}%
where 
\begin{equation*}
\omega =\frac{\xi _{1}-\xi _{2}}{|\xi _{1}-\xi _{2}|}\,,\quad \lambda =\frac{%
\tau ^{\prime }+|\xi _{1}-\xi _{2}|^{2}+|\xi _{2}|^{2}}{2|\xi _{1}-\xi _{2}|}
\end{equation*}%
Now let 
\begin{equation}
H^{\prime }(\tau ^{\prime },\xi _{1},\xi _{2})=\int_{\xi _{2}^{\prime }\in
P^{\prime }(\tau ^{\prime },\xi _{1},\xi _{2})}\frac{d\sigma (\xi
_{2}^{\prime })}{\left\vert \xi _{1}-\xi _{2}-\xi _{2}^{\prime }\right\vert
^{2}\left\vert \xi _{2}^{\prime }\right\vert ^{2}}  \label{E:KMloc7}
\end{equation}%
where the integral is computed with respect to the surface measure on $%
P^{\prime }$.

Since $|\xi _{1}-\xi _{2}|\sim 1$, $|\xi _{2}|\sim 1$, we have the following
reduction 
\begin{equation*}
I^{\prime }(\tau ^{\prime },\xi _{1})\lesssim |\xi _{1}|^{2}\int_{\frac{1}{2}%
\leq |\xi _{2}|\leq 2}H^{\prime }(\tau ^{\prime },\xi _{1},\xi _{2})\,d\xi
_{2}
\end{equation*}%
We now evaluate $H^{\prime }(\tau ^{\prime },\xi _{1},\xi _{2})$. Introduce
polar coordinates $(\rho ,\theta )$ on the plane $P^{\prime }$ with respect
to the \textquotedblleft center\textquotedblright\ $\lambda \omega $, and
note that 
\begin{equation}
\begin{aligned} |\xi_1-\xi_2-\xi_2'|^2 &= | |\xi_1-\xi_2| \omega - \xi_2'|^2
\\ &= | (|\xi_1-\xi_2| - \lambda)\omega - (\xi_2'-\lambda \omega)|^2 \\ &=
(|\xi_1-\xi_2| - \lambda)^2 + |\xi_2'-\lambda \omega|^2 \\ &= (|\xi_1-\xi_2|
- \lambda)^2 + \rho^2\\ &= \alpha^2+\rho^2 \end{aligned}  \label{E:KMloc5}
\end{equation}%
where 
\begin{equation*}
\alpha =|\xi _{1}-\xi _{2}|-\lambda =\frac{|\xi _{1}|^{2}-2\xi _{1}\cdot \xi
_{2}-\tau ^{\prime }}{2|\xi _{1}-\xi _{2}|}
\end{equation*}%
Also,%
\begin{equation}
\left\vert \xi _{2}^{\prime }\right\vert ^{2}=\left\vert \left( \xi
_{2}^{\prime }-\lambda \omega \right) +\lambda \omega \right\vert
^{2}=\left\vert \xi _{2}^{\prime }-\lambda \omega \right\vert ^{2}+\lambda
^{2}=\rho ^{2}+\lambda ^{2}  \label{E:KMloc6}
\end{equation}%
Using \eqref{E:KMloc5} and \eqref{E:KMloc6} in \eqref{E:KMloc7}, 
\begin{equation*}
H^{\prime }(\tau ^{\prime },\xi _{1},\xi _{2})=\int_{0}^{\sqrt{4-\lambda ^{2}%
}}\frac{2\pi \rho \,d\rho }{(\rho ^{2}+\alpha ^{2})(\rho ^{2}+\lambda ^{2})}
\end{equation*}%
The restriction to $0\leq \rho \leq \sqrt{4-\lambda ^{2}}$ arises from the
fact that the plane $P^{\prime }$ must sit within the ball $|\xi
_{2}^{\prime }|\leq 2$. In particular, $H^{\prime }(\tau ,\xi _{1},\xi
_{2})=0$ if $|\lambda |\geq 2$ since then the plane $P^{\prime }$ is located
entirely outside the ball $|\xi _{2}^{\prime }|\leq 2$. Since $|\lambda
|\leq 2$, we have $|\alpha |\leq 3$ and $|\tau ^{\prime }|\leq 10$.

We consider the three cases: (A) $|\lambda| \leq \frac14$ (which implies $%
|\alpha| \geq \frac14$), (B) $|\alpha| \leq \frac14$ (which implies $%
|\lambda| \geq \frac14$), and (C) $|\lambda| \geq \frac14$ and $|\alpha|\geq
\frac14$. Case (C) is the easiest since clearly $|H^{\prime }(\tau^{\prime
}, \xi_1,\xi_2)| \leq C$.

Let us consider case (B). Then 
\begin{equation*}
H^{\prime }(\tau ,\xi _{1},\xi _{2})\lesssim \int_{\rho =0}^{2}\frac{\rho
\,d\rho }{\rho ^{2}+\alpha ^{2}}=\int_{\nu =0}^{\sqrt{2}}\frac{d\nu }{\nu
+\alpha ^{2}}=\log \left( 1+\frac{\sqrt{2}}{\alpha ^{2}}\right)
\end{equation*}%
Substituting back into $I^{\prime }$, 
\begin{equation*}
I^{\prime }(\tau ^{\prime },\xi _{1})\lesssim |\xi _{1}|^{2}\int_{|\xi
_{2}|\leq 2}\log \left( 1+\frac{\sqrt{2}}{\alpha ^{2}}\right) \,d\xi _{2}
\end{equation*}%
Since $|\alpha |\leq \sqrt{3}$, it follows that\footnote{%
The first step is simply: if $x\geq \delta >0$, then $\log (1+x)\leq \log
x+\log (1+\frac{1}{\delta })$. The second step uses that $|\xi _{1}-\xi
_{2}|\sim 1$, which follows since $|\xi _{1}|\ll 1$ and $|\xi _{2}|\sim 1$.} 
\begin{align*}
\log (1+\frac{\sqrt{2}}{\alpha ^{2}})& \leq c+|\log |\alpha || \\
& \leq c+|\log |(|\xi _{1}|^{2}-2\xi _{1}\cdot \xi _{2}-\tau ^{\prime })|| \\
& =c+|\log 2|\xi _{1}\cdot (\xi _{2}-\frac{1}{2}\xi _{1}+\frac{\tau ^{\prime
}\xi _{1}}{2|\xi _{1}|^{2}})| \\
& =c+|\log |\xi _{1}\cdot (\xi _{2}-\frac{1}{2}\xi _{1}+\frac{\tau ^{\prime
}\xi _{1}}{2|\xi _{1}|^{2}})|
\end{align*}%
Hence 
\begin{equation*}
I^{\prime }(\tau ^{\prime },\xi _{1})\lesssim |\xi _{1}|^{2}\left(
1+\int_{|\xi _{2}|\leq 2}|\log |\xi _{1}\cdot (\xi _{2}-\frac{1}{2}\xi _{1}+%
\frac{\tau ^{\prime }\xi _{1}}{2|\xi _{1}|^{2}})|\,d\xi _{2}\right)
\end{equation*}%
Denoting by $B(\mu ,r)$ the ball of center $\mu $ and radius $r$, the
substitution $\xi _{2}\mapsto \xi _{2}+\frac{1}{2}\xi _{1}-\frac{\tau
^{\prime }\xi _{1}}{2|\xi _{1}|^{2}}$ yields, with $\mu =\frac{1}{2}\xi _{1}-%
\frac{\tau ^{\prime }\xi _{1}}{2|\xi _{1}|^{2}}$, 
\begin{align*}
I^{\prime }(\tau ^{\prime },\xi _{1})& \lesssim |\xi _{1}|^{2}\left(
1+\int_{B(\mu ,2)}|\log |\xi _{1}\cdot \xi _{2}||\,d\xi _{2}\right) \\
& \lesssim |\xi _{1}|^{2}\left( \log |\xi _{1}|^{-1}+\int_{B(\mu ,2)}|\log |%
\frac{\xi _{1}}{|\xi _{1}|}\cdot \xi _{2}||\,d\xi _{2}\right)
\end{align*}%
By rotating coordinates so that $\frac{\xi _{1}}{|\xi _{1}|}=(1,0,0)$, and
letting $\mu ^{\prime }$ denote the corresponding rotation of $\mu $, 
\begin{equation*}
I^{\prime }(\tau ^{\prime },\xi _{1})\lesssim |\xi _{1}|^{2}\left( \log |\xi
_{1}|^{-1}+\int_{B(\mu ^{\prime },2)}|\log |(\xi _{2})_{1}|\,d\xi _{2}\right)
\end{equation*}%
where $(\xi _{2})_{1}$ denotes the first coordinate of the vector $\xi _{2}$%
. Since $|\tau ^{\prime }|\leq 10$, it follows that $|\mu ^{\prime
}|\lesssim |\xi _{1}|^{-1}$ and we finally obtain 
\begin{equation*}
I^{\prime }(\tau ^{\prime },\xi _{1})\lesssim |\xi _{1}|^{2}\log |\xi
_{1}|^{-1}
\end{equation*}%
as claimed, completing Case (B).

Case (A) is similar except that we begin with the bound 
\begin{equation*}
H^{\prime }(\tau^{\prime },\xi_1,\xi_2) \lesssim \int_{\rho=0}^{2} \frac{%
2\pi \rho \, d\rho }{ \rho^2+\lambda^2}
\end{equation*}
This completes the proof of \eqref{E:KMloc4}.

Next, we prove \eqref{E:KMloc3}. In the integral defining $I(\tau^{\prime
},\xi_1)$, we have the restriction $\frac12 \leq |\xi_2^{\prime }| \leq 2$
and $|\xi_2| \leq 2$. Note that if $\frac14\leq |\xi_2| \leq 2$, then the
argument above that provided the bound for $I^{\prime }(\tau^{\prime
},\xi_1) $ applies. Hence it suffices to restrict to $|\xi_2| \leq \frac14$,
from which it follows that $|\xi_1-\xi_2-\xi_2^{\prime }| \sim 1$.

Begin by carrying out the $\xi_{2}^{\prime }$ integral to obtain 
\begin{equation}  \label{E:KMloc20}
I(\tau ^{\prime },\xi _{1})=\frac{1}{2}|\xi _{1}|^{2}\int_{|\xi _{2}|\leq 2}%
\frac{H(\tau ^{\prime },\xi _{1},\xi _{2})}{|\xi _{1}-\xi _{2}||\xi _{2}|^{2}%
}\;d\xi _{2}
\end{equation}
where $H(\tau ^{\prime },\xi _{1},\xi _{2})$ is defined as follows. Let $P$
be the truncated plane defined by 
\begin{equation*}
P(\tau ^{\prime },\xi _{1},\xi _{2})=\{\;\xi _{2}^{\prime }\in \mathbb{R}%
^{3}\;|\;(\xi _{2}^{\prime }-\lambda \omega )\cdot \omega =0\,,\; \frac12
\leq |\xi _{2}^{\prime }|\leq 2\;\}
\end{equation*}
where 
\begin{equation*}
\omega =\frac{\xi _{1}-\xi _{2}}{|\xi _{1}-\xi _{2}|}\,,\quad \lambda =\frac{%
\tau ^{\prime }+|\xi _{1}-\xi _{2}|^{2}+|\xi _{2}|^{2}}{2|\xi _{1}-\xi _{2}|}
\end{equation*}
Now let 
\begin{equation*}
H(\tau ^{\prime },\xi _{1},\xi _{2})=\int_{\xi _{2}^{\prime }\in P(\tau
^{\prime },\xi _{1},\xi _{2})}\frac{d\sigma (\xi _{2}^{\prime })}{\left\vert
\xi _{1}-\xi _{2}-\xi _{2}^{\prime }\right\vert ^{2}\left\vert \xi
_{2}^{\prime }\right\vert ^{2}}
\end{equation*}
where the integral is computed with respect to the surface measure on $P$.
Since $|\xi_1-\xi_2-\xi_2^{\prime }| \sim 1$ and $|\xi_2^{\prime }|\sim 1$,
we obtain $H(\tau^{\prime },\xi_1,\xi_2) \leq C$. Substituting into %
\eqref{E:KMloc20}, we obtain 
\begin{align*}
I(\tau^{\prime },\xi_1) &\lesssim |\xi_1|^2 \int_{|\xi_2| \leq \frac14} 
\frac{ d\xi_2 }{|\xi_1-\xi_2||\xi_2|^2} \\
&\lesssim |\xi_1|^2 \left(\int_{|\xi_2|\leq 2|\xi_1|} \frac{ d\xi_2 }{%
|\xi_1-\xi_2||\xi_2|^2} + \int_{2|\xi_1|\leq |\xi_2|\leq \frac14} \frac{
d\xi_2 }{|\xi_1-\xi_2||\xi_2|^2} \right)
\end{align*}
In the first integral, we change variables $\xi_2= |\xi_1|\eta$, and in the
second integral, we use the bound $|\xi_1-\xi_2|^{-1} \leq 2|\xi_2|^{-1}$ to
obtain 
\begin{equation*}
\lesssim |\xi_1|^2 \left(\int_{|\eta|\leq 2} \frac{d\eta}{|\frac{\xi_1}{%
|\xi_1|} - \eta| |\eta|^2} + \int_{2|\xi_1| \leq |\xi_2| \leq \frac14} \frac{%
d\xi_2}{|\xi_2|^3}\right) \lesssim |\xi_1|^2 \log |\xi_1|^{-1}
\end{equation*}
This completes the proof of \eqref{E:KMloc3}.
\end{proof}

\begin{lemma}
\label{Lemma:LocalizedKMWithX_b}For each $\varepsilon >0$, there is a $%
C_{\varepsilon }$ independent of $M_{k},j,k$, and $N$ such that 
\begin{equation*}
\begin{aligned}
\indentalign \Vert R^{(k)}P_{\leqslant M_{k}}^{(k)}B_{N,j,k+1}\alpha ^{(k+1)}\Vert
_{L_{t}^{2}L_{\mathbf{x},\mathbf{x}^{\prime }}^{2}} \\
&\leqslant C_{\varepsilon
}\sum_{M_{k+1}\geqslant M_{k}}\left( \frac{M_{k}}{M_{k+1}}\right)
^{1-\varepsilon }\left\Vert R^{(k+1)}P_{\leqslant M_{k+1}}^{(k+1)}\alpha
^{(k+1)}\right\Vert _{X_{\frac{1}{2}+}^{(k)}}.
\end{aligned}
\end{equation*}
where the sum on the right is in $M_{k+1}$, over dyads such that $%
M_{k+1}\geqslant M_{k}$.
\end{lemma}

\begin{proof}
The proof is exactly the same as deducing Lemma \ref%
{Lemma:KMEstimateInWithX_b} from Lemma \ref{Lemma:KMOriginalEstimateWithVN}.
We include the proof for completeness. Let 
\begin{equation*}
f_{\tau }^{(k+1)}(\mathbf{x}_{k+1},\mathbf{x}_{k+1}^{\prime })=\mathcal{F}%
_{t\mapsto \tau }(U^{(k+1)}(-t)\alpha ^{(k+1)}(t,\mathbf{x}_{k+1},\mathbf{x}%
_{k+1}^{\prime }))
\end{equation*}%
where $\mathcal{F}_{t\mapsto \tau }$ denotes the Fourier transform in $%
t\mapsto \tau $. Then 
\begin{equation*}
\alpha ^{(k+1)}(t,\mathbf{x}_{k+1},\mathbf{x}_{k+1}^{\prime })=\int_{\tau
}e^{it\tau }U^{(k+1)}(t)f^{(k+1)}(\mathbf{x}_{k+1},\mathbf{x}_{k+1}^{\prime
})\,d\tau
\end{equation*}%
By Minkowski's inequality 
\begin{equation*}
\Vert R^{(k)}P_{\leqslant M_{k}}^{(k)}B_{N,j,k+1}\alpha ^{(k+1)}\Vert
_{L_{t}^{2}L_{\mathbf{x},\mathbf{x}^{\prime }}^{2}}\leq \int_{\tau }\Vert
R^{(k)}P_{\leqslant M_{k}}^{(k)}B_{N,j,k+1}U^{(k+1)}(t)f^{(k+1)}\Vert
_{L_{t}^{2}L_{\mathbf{x},\mathbf{x}^{\prime }}^{2}}\,d\tau
\end{equation*}%
By Lemma \ref{Lemma:LocalizedKM}, 
\begin{equation*}
\leq C_{\varepsilon }\sum_{M_{k+1}\geqslant M_{k}}\left( \frac{M_{k}}{M_{k+1}%
}\right) ^{1-\varepsilon }\int_{\tau }\Vert R^{(k+1)}P_{\leqslant
M_{k+1}}^{(k+1)}f^{(k+1)}\Vert _{L_{\mathbf{x},\mathbf{x}^{\prime
}}^{2}}\,d\tau
\end{equation*}%
For any $b>\frac{1}{2}$, we write $1=\langle \tau \rangle ^{-b}\langle \tau
\rangle ^{b}$ and apply Cauchy-Schwarz in $\tau $ to obtain 
\begin{eqnarray*}
&\leq &C_{\varepsilon }\sum_{M_{k+1}\geqslant M_{k}}\left( \frac{M_{k}}{%
M_{k+1}}\right) ^{1-\varepsilon }\Vert \langle \tau \rangle
^{b}R^{(k+1)}P_{\leqslant M_{k+1}}^{(k+1)}f^{(k+1)}\Vert _{L_{\tau ,\mathbf{x%
},\mathbf{x}^{\prime }}^{2}} \\
&=&C_{\varepsilon }\sum_{M_{k+1}\geqslant M_{k}}\left( \frac{M_{k}}{M_{k+1}}%
\right) ^{1-\varepsilon }\left\Vert R^{(k+1)}P_{\leqslant
M_{k+1}}^{(k+1)}\alpha ^{(k+1)}\right\Vert _{X_{\frac{1}{2}+}^{(k)}}.
\end{eqnarray*}
\end{proof}

\subsection{A Strichartz estimate}

\begin{lemma}
\label{Lemma:TheStrichartzEstimate}Assume $\gamma ^{(k)}(t,\mathbf{x}_{k}; 
\mathbf{x}_{k}^{\prime })$ satisfies the symmetric condition %
\eqref{condition:symmetric}. Let 
\begin{equation}
\beta ^{(k)}(t,\mathbf{x}_{k};\mathbf{x}_{k}^{\prime })=V(x_{i}-x_{j})\gamma
^{(k)}(t,\mathbf{x}_{k};\mathbf{x}_{k}^{\prime })  \label{E:beta-def}
\end{equation}
Then we have the estimates: 
\begin{eqnarray}
\Vert \beta ^{(k)}\Vert _{X_{-\frac{1}{2}+}^{(k)}} &\lesssim &\Vert V\Vert
_{L_{x}^{\frac{6}{5}+}}\Vert \langle \nabla _{x_{i}}\rangle \langle \nabla
_{x_{j}}\rangle \gamma ^{(k)}\Vert _{L_{t}^{2}L_{x,x^{\prime }}^{2}},
\label{Strichartzestimate:TheDerivativeOne} \\
\Vert \beta ^{(k)}\Vert _{X_{-\frac{1}{2}+}^{(k)}} &\lesssim &\Vert V\Vert
_{L_{x}^{3+}}\Vert \gamma ^{(k)}\Vert _{L_{t}^{2}L_{x,x^{\prime }}^{2}},
\label{Strichartzestimate:TheNoDerivativeOne} \\
\Vert \beta ^{(k)}\Vert _{X_{-\frac{1}{2}+}^{(k)}} &\lesssim &\Vert V\Vert
_{L_{x}^{2+}}\Vert \langle \nabla _{x_{i}}\rangle ^{\frac{1}{2}}\gamma
^{(k)}\Vert _{L_{t}^{2}L_{x,x^{\prime }}^{2}}.
\label{Strichartzestimate:TheHalfDerivativeOne}
\end{eqnarray}
\end{lemma}

\begin{proof}
It suffices to prove Lemma \ref{Lemma:TheStrichartzEstimate} for $k=2$.
Since we will need to deal with Fourier transforms in only selected
coordinates, we introduce the following notation: $\mathcal{F}_{0}$ denotes
Fourier transform in $t$, $\mathcal{F}_{j}$ denotes Fourier transform in $%
x_{j}$, and $\mathcal{F}_{j^{\prime }}$ denotes Fourier transform in $%
x_{j}^{\prime }$. Fourier transforms in multiple coordinates will denoted as
combined subscripts -- for example, $\mathcal{F}_{01^{\prime }}=\mathcal{F}%
_{0}\mathcal{F}_{1^{\prime }}$ denotes the Fourier transform in $t$ and $%
x_{1}^{\prime }$.\footnote{%
We are going to apply the endpoint Strichartz estimate on the
non-transformed coordinates. We do not know the origin of such a technique,
although it was also used by the first author in \cite[Lemma 6]{Chen2ndOrder}%
.}

We start by splitting $\gamma ^{(2)}$ into the piece where $\left\vert \xi
_{1}\right\vert \geqslant \left\vert \xi _{2}\right\vert $ and the piece
where $\left\vert \xi _{2}\right\vert \geqslant \left\vert \xi
_{1}\right\vert $: 
\begin{equation*}
\gamma ^{(2)}=\gamma _{\left\vert \xi _{1}\right\vert \geqslant \left\vert
\xi _{2}\right\vert }^{(2)}+\gamma _{\left\vert \xi _{2}\right\vert
\geqslant \left\vert \xi _{1}\right\vert }^{(2)}.
\end{equation*}%
The below represents the treatment of 
\begin{equation*}
\beta _{\left\vert \xi _{2}\right\vert \geqslant \left\vert \xi
_{1}\right\vert }^{(2)}=V(x_{1}-x_{2})\gamma _{\left\vert \xi
_{2}\right\vert \geqslant \left\vert \xi _{1}\right\vert }^{(2)}
\end{equation*}%
since the $\left\vert \xi _{1}\right\vert \geqslant \left\vert \xi
_{2}\right\vert $ case is similar. Let $T$ denote the translation operator 
\begin{equation*}
(Tf)(x_{1},x_{2})=f(x_{1}+x_{2},x_{2})
\end{equation*}%
Suppressing the $x_{1}^{\prime }$, $x_{2}^{\prime }$ dependence, we have 
\begin{equation}
(\mathcal{F}_{12}T\beta _{\left\vert \xi _{2}\right\vert \geqslant
\left\vert \xi _{1}\right\vert }^{(2)})(t,\xi _{1},\xi _{2})=(\mathcal{F}%
_{12}\beta _{\left\vert \xi _{2}\right\vert \geqslant \left\vert \xi
_{1}\right\vert }^{(2)})(t,\xi _{1},\xi _{2}-\xi _{1})  \label{E:Fourier1}
\end{equation}%
Also 
\begin{equation}
e^{-2it\xi _{1}\cdot \xi _{2}}(\mathcal{F}_{12}T\beta _{\left\vert \xi
_{2}\right\vert \geqslant \left\vert \xi _{1}\right\vert }^{(2)})(t,\xi
_{1},\xi _{2})=\mathcal{F}_{1}\big[(\mathcal{F}_{2}T\beta _{\left\vert \xi
_{2}\right\vert \geqslant \left\vert \xi _{1}\right\vert
}^{(2)})(t,x_{1}-2t\xi _{2},\xi _{2})\big](\xi _{1})  \label{E:Fourier2}
\end{equation}%
Now 
\begin{equation}
\begin{aligned} \hspace{0.3in}&\hspace{-0.3in} (\mathcal{F}_{012}\beta
_{\left\vert \xi _{2}\right\vert \geqslant \left\vert \xi _{1}\right\vert
}^{(2)})(\tau - |\xi_2|^2 + 2\xi_1\cdot\xi_2, \xi_1,\xi_2-\xi_1) \\ &=
(\mathcal{F}_{012}T\beta _{\left\vert \xi _{2}\right\vert \geqslant
\left\vert \xi _{1}\right\vert }^{(2)})(\tau - |\xi_2|^2 + 2\xi_1\cdot\xi_2,
\xi_1,\xi_2) && \text{by }\eqref{E:Fourier1}\\ &= \mathcal{F}_0\big[
e^{it|\xi_2|^2} e^{-2it\xi_1\cdot\xi_2} (\mathcal{F}_{12}T\beta _{\left\vert
\xi _{2}\right\vert \geqslant \left\vert \xi _{1}\right\vert
}^{(2)})(t,\xi_1,\xi_2)\big](\tau) \\ &= \mathcal{F}_0\big[ e^{it|\xi_2|^2}
\mathcal{F}_1 \big[ (\mathcal{F}_2T\beta _{\left\vert \xi _{2}\right\vert
\geqslant \left\vert \xi _{1}\right\vert }^{(2)})(t, x_1-2t\xi_2, \xi_2)
\big](\xi_1) \big](\tau) && \text{by } \eqref{E:Fourier2} \\ &=
\mathcal{F}_{01} \big[ e^{it|\xi_2|^2}(\mathcal{F}_2T\beta _{\left\vert \xi
_{2}\right\vert \geqslant \left\vert \xi _{1}\right\vert }^{(2)})(t,
x_1-2t\xi_2, \xi_2) \big] (\tau,\xi_1) \end{aligned}  \label{E:Fourier3}
\end{equation}%
By changing variables $\xi _{2}\mapsto \xi _{2}-\xi _{1}$ and then changing $%
\tau \mapsto \tau -|\xi _{2}|^{2}+2\xi _{1}\cdot \xi _{2}$, we obtain 
\begin{align*}
& \Vert \beta _{\left\vert \xi _{2}\right\vert \geqslant \left\vert \xi
_{1}\right\vert }^{(2)}\Vert _{X_{-\frac{1}{2}+}^{(2)}} \\
& =\Vert \left( \beta _{\left\vert \xi _{2}\right\vert \geqslant \left\vert
\xi _{1}\right\vert }^{(2)}\right) \widehat{\;}(\tau ,\xi _{1},\xi _{2},\xi
_{1}^{\prime },\xi _{2}^{\prime })\langle \tau +\left\vert \xi
_{1}\right\vert ^{2}+\left\vert \xi _{2}\right\vert ^{2}-\left\vert \xi
_{1}^{\prime }\right\vert ^{2}-\left\vert \xi _{2}^{\prime }\right\vert
^{2}\rangle ^{-\frac{1}{2}+}\Vert _{L_{\tau \xi _{1}\xi _{2}\xi _{1}^{\prime
}\xi _{2}^{\prime }}^{2}} \\
& =\Big\| \left( \beta _{\left\vert \xi _{2}\right\vert \geqslant \left\vert
\xi _{1}\right\vert }^{(2)}\right) \widehat{\;}(\tau -|\xi _{2}|^{2}+2\xi
_{1}\cdot \xi _{2},\xi _{1},\xi _{2}-\xi _{1},\xi _{1}^{\prime },\xi
_{2}^{\prime })\\
& \qquad \qquad \times \langle \tau +2\left\vert \xi _{1}\right\vert ^{2}-\left\vert
\xi _{1}^{\prime }\right\vert ^{2}-\left\vert \xi _{2}^{\prime }\right\vert
^{2}\rangle ^{-\frac{1}{2}+}\Big\|_{L_{\tau \xi _{1}\xi _{2}\xi _{1}^{\prime
}\xi _{2}^{\prime }}^{2}}
\end{align*}%
Applying the the dual Strichartz (see \eqref{E:dual-Strichartz} below), the
above is bounded by 
\begin{equation*}
\lesssim \Vert \mathcal{F}_{01}^{-1}\Big[(\mathcal{F}_{012}\beta
_{\left\vert \xi _{2}\right\vert \geqslant \left\vert \xi _{1}\right\vert
}^{(2)})(\tau -|\xi _{2}|^{2}+2\xi _{1}\cdot \xi _{2},\xi _{1},\xi _{2}-\xi
_{1})\Big](t,x_{1})\Vert _{L_{\xi _{2}}^{2}L_{t}^{2}L_{x_{1}}^{\frac{6}{5}%
+}L_{x_{1}^{\prime }x_{2}^{\prime }}^{2}}
\end{equation*}%
Utilizing \eqref{E:Fourier3}, the above is equal to 
\begin{equation*}
=\Vert (\mathcal{F}_{2}T\beta _{\left\vert \xi _{2}\right\vert \geqslant
\left\vert \xi _{1}\right\vert }^{(2)})(t,x_{1}-2t\xi _{2},\xi _{2})\Vert
_{L_{t}^{2}L_{\xi _{2}}^{2}L_{x_{1}}^{\frac{6}{5}+}L_{x_{1}^{\prime
}x_{2}^{\prime }}^{2}}
\end{equation*}%
Change variable in $x_{1}\mapsto x_{1}+2t\xi _{2}$ to obtain 
\begin{equation*}
=\Vert (\mathcal{F}_{2}T\beta _{\left\vert \xi _{2}\right\vert \geqslant
\left\vert \xi _{1}\right\vert }^{(2)})(t,x_{1},\xi _{2})\Vert
_{L_{t}^{2}L_{\xi _{2}}^{2}L_{x_{1}}^{\frac{6}{5}+}L_{x_{1}^{\prime
}x_{2}^{\prime 2}}^{2}}
\end{equation*}%
Now note that from \eqref{E:beta-def}, we have 
\begin{equation*}
(\mathcal{F}_{2}T\beta _{\left\vert \xi _{2}\right\vert \geqslant \left\vert
\xi _{1}\right\vert }^{(2)})(t,x_{1},\xi _{2})=V(x_{1})(\mathcal{F}%
_{2}T\gamma _{\left\vert \xi _{2}\right\vert \geqslant \left\vert \xi
_{1}\right\vert }^{(2)})(t,x_{1},\xi _{2})
\end{equation*}%
It follows that 
\begin{align}
\hspace{0.3in}& \hspace{-0.3in}\Vert (\mathcal{F}_{2}T\beta _{\left\vert \xi
_{2}\right\vert \geqslant \left\vert \xi _{1}\right\vert
}^{(2)})(t,x_{1},\xi _{2})\Vert _{L_{t}^{2}L_{\xi _{2}}^{2}L_{x_{1}}^{\frac{6%
}{5}+}L_{x_{1}^{\prime }x_{2}^{\prime }}^{2}}  \notag \\
& =\left\Vert V(x_{1})\Big(\Vert (\mathcal{F}_{2}T\gamma _{\left\vert \xi
_{2}\right\vert \geqslant \left\vert \xi _{1}\right\vert
}^{(2)})(t,x_{1},\xi _{2})\Vert _{L_{x_{1}^{\prime }x_{2}^{\prime }}^{2}}%
\Big)\right\Vert _{L_{t}^{2}L_{\xi _{2}}^{2}L_{x_{1}}^{\frac{6}{5}+}}  \notag
\\
& \leq \Vert V\Vert _{L^{\frac{6}{5}+}}\Vert (\mathcal{F}_{2}T\gamma
_{\left\vert \xi _{2}\right\vert \geqslant \left\vert \xi _{1}\right\vert
}^{(2)})(t,x_{1},\xi _{2})\Vert _{L_{t}^{2}L_{\xi _{2}}^{2}L_{x_{1}}^{\infty
}L_{x_{1}^{\prime }x_{2}^{\prime }}^{2}}  \label{Strichartz:DifferentHolder}
\\
& \leq \Vert V\Vert _{L^{\frac{6}{5}+}}\Vert (\mathcal{F}_{2}T\gamma
_{\left\vert \xi _{2}\right\vert \geqslant \left\vert \xi _{1}\right\vert
}^{(2)})(t,x_{1},\xi _{2})\Vert _{L_{t}^{2}L_{\xi _{2}x_{1}^{\prime
}x_{2}^{\prime }}^{2}L_{x_{1}}^{\infty }}  \notag
\end{align}%
Apply Sobolev in $x_{1}$: 
\begin{equation*}
\lesssim \Vert V\Vert _{L^{\frac{6}{5}+}}\Vert \langle \nabla
_{x_{1}}\rangle ^{2}(\mathcal{F}_{2}T\gamma _{\left\vert \xi _{2}\right\vert
\geqslant \left\vert \xi _{1}\right\vert }^{(2)})(t,x_{1},\xi _{2})\Vert
_{L_{t}^{2}L_{\xi _{2}x_{1}^{\prime }x_{2}^{\prime }}^{2}L_{x_{1}}^{2}}
\end{equation*}%
Move the $d\xi _{2}dx_{1}^{\prime }dx_{2}^{\prime }$ integration to the
inside and apply Plancherel in $\xi _{2}\mapsto x_{2}$ to obtain 
\begin{eqnarray*}
&=&\Vert V\Vert _{L^{\frac{6}{5}+}}\Vert \langle \nabla _{x_{1}}\rangle
^{2}T\gamma _{\left\vert \xi _{2}\right\vert \geqslant \left\vert \xi
_{1}\right\vert }^{(2)}\Vert _{L_{t}^{2}L_{\mathbf{x,x}^{\prime }}^{2}} \\
&=&\Vert V\Vert _{L^{\frac{6}{5}+}}\Vert \langle \nabla _{x_{1}}\rangle
^{2}\gamma _{\left\vert \xi _{2}\right\vert \geqslant \left\vert \xi
_{1}\right\vert }^{(2)}\Vert _{L_{t}^{2}L_{\mathbf{x,x}^{\prime }}^{2}}.
\end{eqnarray*}%
Recall that the $\xi _{2}$ frequency dominates in $\gamma _{\left\vert \xi
_{2}\right\vert \geqslant \left\vert \xi _{1}\right\vert }^{(2)}$, and thus
this is bounded above by 
\begin{eqnarray*}
&\lesssim &\Vert V\Vert _{L^{\frac{6}{5}+}}\Vert \langle \nabla
_{x_{1}}\rangle \langle \nabla _{x_{2}}\rangle \gamma _{\left\vert \xi
_{2}\right\vert \geqslant \left\vert \xi _{1}\right\vert }^{(2)}(t,\mathbf{x}%
_{2},\mathbf{x}_{2}^{\prime })\Vert _{L_{t}^{2}L_{\mathbf{x,x}^{\prime
}}^{2}} \\
&\lesssim &\Vert V\Vert _{L^{\frac{6}{5}+}}\Vert \langle \nabla
_{x_{1}}\rangle \langle \nabla _{x_{2}}\rangle \gamma ^{(2)}(t,\mathbf{x}%
_{2},\mathbf{x}_{2}^{\prime })\Vert _{L_{t}^{2}L_{\mathbf{x,x}^{\prime
}}^{2}}.
\end{eqnarray*}%
This proves estimate (\ref{Strichartzestimate:TheDerivativeOne}). Using H%
\"{o}lder exponents $(3$+$,2,\frac{6}{5}$+$)$ and $(2$+$,3,\frac{6}{5}$+$)$
in (\ref{Strichartz:DifferentHolder}) yields estimates (\ref%
{Strichartzestimate:TheNoDerivativeOne}) and (\ref%
{Strichartzestimate:TheHalfDerivativeOne}). Their proofs are easier in the
sense that there is no need to split $\gamma ^{(2)}.$

It remains to prove the following dual Strichartz estimate (here $\sigma
^{(2)}(t,x_{1},x_{1}^{\prime },x_{2}^{\prime })$, note that the $x_{2}$
coordinate is missing): 
\begin{equation}
\Vert \langle \tau +2\left\vert \xi _{1}\right\vert ^{2}-\left\vert \xi
_{1}^{\prime }\right\vert ^{2}-\left\vert \xi _{2}^{\prime }\right\vert
^{2}\rangle ^{-\frac{1}{2}+}\hat{\sigma}^{(2)}(\tau ,\xi _{1},\xi
_{1}^{\prime },\xi _{2}^{\prime })\Vert _{L_{\tau }^{2}L_{\xi _{1}\xi
_{1}^{\prime }\xi _{2}^{\prime }}^{2}}\lesssim \Vert \sigma ^{(2)}\Vert
_{L_{t}^{2}L_{x_{1}}^{\frac{6}{5}+}L_{x_{1}^{\prime }x_{2}^{\prime }}^{2}}
\label{E:dual-Strichartz}
\end{equation}%
The estimate \eqref{E:dual-Strichartz} is dual to the equivalent estimate 
\begin{equation}
\Vert \sigma ^{(2)}\Vert _{L_{t}^{2}L_{x_{1}}^{6-}L_{x_{1}^{\prime
}x_{2}^{\prime }}^{2}}\lesssim \Vert \langle \tau +2\left\vert \xi
_{1}\right\vert ^{2}-\left\vert \xi _{1}^{\prime }\right\vert
^{2}-\left\vert \xi _{2}^{\prime }\right\vert ^{2}\rangle ^{\frac{1}{2}-}%
\hat{\sigma}^{(2)}(\tau ,\xi _{1},\xi _{1}^{\prime },\xi _{2}^{\prime
})\Vert _{L_{\tau }^{2}L_{\xi _{1}\xi _{1}^{\prime }\xi _{2}^{\prime }}^{2}}
\label{E:StrichartzA}
\end{equation}%
To prove \eqref{E:StrichartzA}, we prove 
\begin{equation}
\Vert \sigma ^{(2)}\Vert _{L_{t}^{2}L_{x_{1}}^{6}L_{x_{1}^{\prime
}x_{2}^{\prime }}^{2}}\lesssim \Vert \langle \tau +2\left\vert \xi
_{1}\right\vert ^{2}-\left\vert \xi _{1}^{\prime }\right\vert
^{2}-\left\vert \xi _{2}^{\prime }\right\vert ^{2}\rangle ^{\frac{1}{2}+}%
\hat{\sigma}^{(2)}(\tau ,\xi _{1},\xi _{1}^{\prime },\xi _{2}^{\prime
})\Vert _{L_{\tau }^{2}L_{\xi _{1}\xi _{1}^{\prime }\xi _{2}^{\prime }}^{2}}
\label{E:StrichartzB}
\end{equation}%
The estimate \eqref{E:StrichartzA} follows from the interpolation of %
\eqref{E:StrichartzB} and the trivial equality 
\begin{equation*}
\Vert \sigma ^{(2)}\Vert _{L_{t}^{2}L_{x_{1}}^{2}L_{x_{1}^{\prime
}x_{2}^{\prime }}^{2}}=\Vert \langle \tau +2\left\vert \xi _{1}\right\vert
^{2}-\left\vert \xi _{1}^{\prime }\right\vert ^{2}-\left\vert \xi
_{2}^{\prime }\right\vert ^{2}\rangle ^{0}\hat{\sigma}^{(2)}(\tau ,\xi
_{1},\xi _{1}^{\prime },\xi _{2}^{\prime })\Vert _{L_{\tau }^{2}L_{\xi
_{1}\xi _{1}^{\prime }\xi _{2}^{\prime }}^{2}}
\end{equation*}%
Thus proving \eqref{E:dual-Strichartz} is reduced to proving %
\eqref{E:StrichartzB}, which we do now. Let 
\begin{equation}
\phi _{\tau }(x_{1},x_{1}^{\prime },x_{2})\overset{\mathrm{def}}{=}\mathcal{F%
}_{0}[U^{1}(-2t)U^{1^{\prime }}(-t)U^{2^{\prime }}(-t)\sigma
^{(2)}(t,x_{1},x_{1}^{\prime },x_{2}^{\prime })](\tau )  \label{E:phi-def}
\end{equation}%
Then note $\phi _{\tau }$ is independent of $t$ and 
\begin{equation*}
\sigma ^{(2)}(t,x_{1},x_{1}^{\prime },x_{2}^{\prime })=\int e^{it\tau
}U^{1}(2t)U^{1^{\prime }}(t)U^{2^{\prime }}(t)\phi _{\tau
}(x_{1},x_{1}^{\prime },x_{2}^{\prime })d\tau
\end{equation*}%
Thus 
\begin{align*}
\Vert \sigma ^{(2)}\Vert _{L_{t}^{2}L_{x_{1}}^{6}L_{x_{1}^{\prime
}x_{2}^{\prime }}^{2}}& \lesssim \int_{\tau }\Vert U^{1^{\prime
}}(t)U^{2^{\prime }}(t)U^{1}(2t)\phi _{\tau }(x_{1},x_{1}^{\prime
},x_{2}^{\prime })\Vert _{L_{t}^{2}L_{x_{1}}^{6}L_{x_{1}^{\prime
}x_{2}^{\prime }}^{2}}\,d\tau \\
& \lesssim \int_{\tau }\Vert U^{1}(2t)\phi _{\tau }(x_{1},x_{1}^{\prime
},x_{2}^{\prime })\Vert _{L_{t}^{2}L_{x_{1}}^{6}L_{x_{1}^{\prime
}x_{2}^{\prime }}^{2}}\,d\tau \\
& \lesssim \int_{\tau }\Vert U^{1}(2t)\phi _{\tau }(x_{1},x_{1}^{\prime
},x_{2}^{\prime })\Vert _{L_{x_{1}^{\prime }x_{2}^{\prime
}}^{2}L_{t}^{2}L_{x_{1}}^{6}}\,d\tau
\end{align*}%
Now apply Keel-Tao \cite{Keel-Tao} endpoint Strichartz estimate to obtain 
\begin{align*}
& \lesssim \int_{\tau }\Vert \phi _{\tau }(x_{1},x_{1}^{\prime
},x_{2}^{\prime })\Vert _{L_{x_{1}^{\prime }x_{2}^{\prime
}}^{2}L_{x_{1}}^{2}}\,d\tau \\
& \lesssim \Vert \langle \tau \rangle ^{\frac{1}{2}+}\phi _{\tau
}(x_{1},x_{1}^{\prime },x_{2}^{\prime })\Vert _{L_{\tau
}^{2}L_{x_{1}x_{1}^{\prime }x_{2}^{\prime }}^{2}}
\end{align*}%
It follows from \eqref{E:phi-def} that 
\begin{equation*}
=\Vert \langle \tau +2\left\vert \xi _{1}\right\vert ^{2}-\left\vert \xi
_{1}^{\prime }\right\vert ^{2}-\left\vert \xi _{2}^{\prime }\right\vert
^{2}\rangle ^{\frac{1}{2}+}\hat{\sigma}^{(2)}(\tau ,\xi _{1},\xi
_{1}^{\prime },\xi _{2}^{\prime })\Vert _{L_{\tau \xi _{1}\xi _{1}^{\prime
}\xi _{2}^{\prime }}^{2}}
\end{equation*}%
which completes the proof of \eqref{E:StrichartzB}.
\end{proof}

\begin{corollary}
Let 
\begin{equation*}
\beta ^{(k)}(t,\mathbf{x}_{k},\mathbf{x}_{k}^{\prime })=N^{3\beta
-1}V(N^{\beta }(x_{i}-x_{j}))\gamma ^{(k)}(t,\mathbf{x}_{k},\mathbf{x}%
_{k}^{\prime })
\end{equation*}%
Then for $N\geq 1$, we have 
\begin{equation}
\Vert \left\vert \nabla _{x_{i}}\right\vert \left\vert \nabla
_{x_{j}}\right\vert \beta ^{(k)}\Vert _{X_{-\frac{1}{2}+}^{(k)}}\lesssim N^{%
\frac{5}{2}\beta -1}\Vert \langle \nabla _{x_{i}}\rangle \langle \nabla
_{x_{j}}\rangle \gamma ^{(k)}\Vert _{L_{t}^{2}L_{\mathbf{x}\mathbf{x}%
^{\prime }}^{2}}  \label{E:StrCor1}
\end{equation}%
and 
\begin{equation}
\Vert \beta ^{(k)}\Vert _{X_{-\frac{1}{2}+}^{(k)}}\lesssim N^{\frac{1}{2}%
\beta -1}\Vert \langle \nabla _{x_{i}}\rangle \langle \nabla _{x_{j}}\rangle
\gamma ^{(k)}\Vert _{L_{t}^{2}L_{\mathbf{x}\mathbf{x}^{\prime }}^{2}}
\label{E:StrCor2}
\end{equation}%
Consequently, ($R_{\leq M}^{(k)}=P_{\leq M}^{(k)}R^{(k)}$) 
\begin{equation}
\Vert R_{\leq M}^{(k)}\beta ^{(k)}\Vert _{X_{-\frac{1}{2}+}^{(k)}}\lesssim
N^{\frac{1}{2}\beta -1}\min (M^{2},N^{2\beta })\Vert S^{(k)}\gamma
^{(k)}\Vert _{L_{t}^{2}L_{\mathbf{x}\mathbf{x}^{\prime }}^{2}}
\label{E:StrCor3}
\end{equation}
\end{corollary}

\begin{proof}
Estimate \eqref{E:StrCor1} follows by applying either %
\eqref{Strichartzestimate:TheDerivativeOne}, %
\eqref{Strichartzestimate:TheNoDerivativeOne}, or %
\eqref{Strichartzestimate:TheHalfDerivativeOne} according to whether two
derivatives, no derivatives, or one derivative, respectively, lands on $%
N^{3\beta-1}V(N^\beta(x_i-x_j))$.

Estimate \eqref{E:StrCor2} follows by applying %
\eqref{Strichartzestimate:TheDerivativeOne}.

Finally, \eqref{E:StrCor3} follows from \eqref{E:StrCor1} and %
\eqref{E:StrCor2}, as follows. Let 
\begin{equation*}
Q=\prod_{\substack{ 1\leq \ell \leq k  \\ \ell \neq i,j}}\left\vert \nabla
_{x_{\ell }}\right\vert
\end{equation*}%
Then 
\begin{equation*}
\Vert R_{\leq M}^{(k)}\beta ^{(k)}\Vert _{X_{-\frac{1}{2}+}^{(k)}}\leq
M^{2}\Vert Q\beta ^{(k)}\Vert _{X_{-\frac{1}{2}+}^{(k)}}
\end{equation*}%
The $Q$ operator passes directly onto $\gamma ^{(k)}$, and one applies %
\eqref{E:StrCor2} to obtain 
\begin{equation}
\Vert R_{\leq M}^{(k)}\beta ^{(k)}\Vert _{X_{-\frac{1}{2}+}^{(k)}}\lesssim
N^{\frac{1}{2}\beta -1}M^{2}\Vert S^{(k)}\gamma ^{(k)}\Vert _{L_{t}^{2}L_{%
\mathbf{x}\mathbf{x}^{\prime }}^{2}}  \label{E:StrCor4}
\end{equation}%
On the other hand, 
\begin{equation*}
\Vert R_{\leq M}^{(k)}\beta ^{(k)}\Vert _{X_{-\frac{1}{2}+}^{(k)}}\leq \Vert
Q\left\vert \nabla _{x_{i}}\right\vert \left\vert \nabla _{x_{j}}\right\vert
\beta ^{(k)}\Vert _{X_{-\frac{1}{2}+}^{(k)}}
\end{equation*}%
The $Q$ operator passes directly on $\gamma ^{(k)}$, and one applies %
\eqref{E:StrCor1} to obtain 
\begin{equation}
\Vert R_{\leq M}^{(k)}\beta ^{(k)}\Vert _{X_{-\frac{1}{2}+}^{(k)}}\lesssim
N^{\frac{5}{2}\beta -1}\Vert S^{(k)}\gamma ^{(k)}\Vert _{L_{t}^{2}L_{\mathbf{%
x}\mathbf{x}^{\prime }}^{2}}  \label{E:StrCor5}
\end{equation}%
Combining \eqref{E:StrCor4} and \eqref{E:StrCor5}, we obtain %
\eqref{E:StrCor3}.
\end{proof}

%\section{Conclusion}

%In this paper, we have established a positive answer to Conjecture \ref%
%{Conjecture:KM} by Klainerman and Machedon \cite{KlainermanAndMachedon} in
%2008 for $\beta \in (0,2/3).$ This is the first progress in proving
%Conjecture \ref{Conjecture:KM} for self-interaction ($\beta >1/3).$
%Moreover, our main theorem (Theorem \ref{THM:MainTHM}) has already fulfilled
%the original intent of \cite{KlainermanAndMachedon}, namely, simplifying the
%uniqueness argument of \cite{E-S-Y2} which deals with $\beta \in (0,3/5).$
%Conjecture \ref{Conjecture:KM} for $\beta \in \lbrack 2/3,1]$ is still open.

\appendix

\section{The topology on the density matrices\label{appendix:ESYTopology}}

In this appendix, we define a topology $\tau _{prod}$ on the density
matrices as was previously done in \cite{E-E-S-Y1, E-Y1,
E-S-Y1,E-S-Y2,E-S-Y5,
E-S-Y3,Kirpatrick,TChenAndNP,ChenAnisotropic,Chen3DDerivation,C-H3Dto2D}.

Denote the spaces of compact operators and trace class operators on $%
L^{2}\left( \mathbb{R}^{3k}\right) $ as $\mathcal{K}_{k}$ and $\mathcal{L}
_{k}^{1}$, respectively. Then $\left( \mathcal{K}_{k}\right) ^{\prime }= 
\mathcal{L}_{k}^{1}$. By the fact that $\mathcal{K}_{k}$ is separable, we
select a dense countable subset $\{J_{i}^{(k)}\}_{i\geqslant 1}\subset 
\mathcal{K}_{k}$ in the unit ball of $\mathcal{K}_{k}$ (so $\Vert
J_{i}^{(k)}\Vert _{\func{op}}\leqslant 1$ where $\left\Vert \cdot
\right\Vert _{\func{op}}$ is the operator norm). For $\gamma ^{(k)},\tilde{
\gamma}^{(k)}\in \mathcal{L}_{k}^{1}$, we then define a metric $d_{k}$ on $%
\mathcal{L}_{k}^{1}$ by 
\begin{equation*}
d_{k}(\gamma ^{(k)},\tilde{\gamma}^{(k)})=\sum_{i=1}^{\infty
}2^{-i}\left\vert \limfunc{Tr}J_{i}^{(k)}\left( \gamma ^{(k)}-\tilde{\gamma}
^{(k)}\right) \right\vert .
\end{equation*}
A uniformly bounded sequence $\gamma _{N}^{(k)}\in \mathcal{L}_{k}^{1}$
converges to $\gamma ^{(k)}\in \mathcal{L}_{k}^{1}$ with respect to the
weak* topology if and only if 
\begin{equation*}
\lim_{N}d_{k}(\gamma _{N}^{(k)},\gamma ^{(k)})=0.
\end{equation*}
For fixed $T>0$, let $C\left( \left[ 0,T\right] ,\mathcal{L}_{k}^{1}\right) $
be the space of functions of $t\in \left[ 0,T\right] $ with values in $%
\mathcal{L}_{k}^{1}$ which are continuous with respect to the metric $d_{k}.$
On $C\left( \left[ 0,T\right] ,\mathcal{L}_{k}^{1}\right) ,$ we define the
metric 
\begin{equation*}
\hat{d}_{k}(\gamma ^{(k)}\left( \cdot \right) ,\tilde{\gamma}^{(k)}\left(
\cdot \right) )=\sup_{t\in \left[ 0,T\right] }d_{k}(\gamma ^{(k)}\left(
t\right) ,\tilde{\gamma}^{(k)}\left( t\right) ).
\end{equation*}
We can then define a topology $\tau _{prod}$ on the space $\oplus
_{k\geqslant 1}C\left( \left[ 0,T\right] ,\mathcal{L}_{k}^{1}\right) $ by
the product of topologies generated by the metrics $\hat{d}_{k}$ on $C\left( %
\left[ 0,T\right] ,\mathcal{L}_{k}^{1}\right) .$

\section{Proof of estimates \eqref{estimate:FreePartofBBGKY} and 
\eqref{estimate:InteractionPartofBBGKY}}

\label{appendix:estimating free and interaction}

\begin{proof}[Proof of Estimate \eqref{estimate:FreePartofBBGKY}]
Utilizing Lemma \ref{lemma:Klainerman-MachedonBoardGameForBBGKY} and
estimate \eqref{Relation:IteratngCollapsingEstimate} to the free part of $%
\gamma _{N}^{(2)}$, we obtain 
\begin{eqnarray*}
&&\left\Vert R^{(k-1)}B_{N,1,k}\text{\textsc{FP}}^{(k,l_{c})}\right\Vert
_{L_{T}^{1}L_{\mathbf{x,x}^{\prime }}^{2}} \\
&\leqslant &CT^{\frac{1}{2}}\left\Vert R^{(k)}\gamma _{N,0}^{(k)}\right\Vert
_{L_{\mathbf{x,x}^{\prime }}^{2}} \\
&&+\sum_{j=1}^{l_{c}}\sum_{m}\left\Vert
R^{(k-1)}B_{N,1,k}\int_{D}J_{N}^{(k,j)}(\underline{t}_{k+j},\mu
_{m})(U^{(k+j)}(t_{k+j})\gamma _{N,0}^{(k+j)})d\underline{t}%
_{k+j}\right\Vert _{L_{T}^{1}L_{\mathbf{x,x}^{\prime }}^{2}} \\
&\leqslant &CT^{\frac{1}{2}}\left\Vert R^{(k)}\gamma _{N,0}^{(k)}\right\Vert
_{L_{\mathbf{x,x}^{\prime }}^{2}} \\
&&+\sum_{j=1}^{l_{c}}\sum_{m}(CT^{\frac{1}{2}})^{j}\left\Vert R^{(k+j-1)}B_{N,\mu _{m}(k+j),k+j}U^{(k+j)}(t_{k+j})\gamma
_{N,0}^{(k+j)}\right\Vert _{L_{T}^{1}L_{\mathbf{x,x}^{\prime }}^{2}} \\
&\leqslant &CT^{\frac{1}{2}}\left\Vert R^{(k)}\gamma _{N,0}^{(k)}\right\Vert
_{L_{\mathbf{x,x}^{\prime }}^{2}}+\sum_{j=1}^{\infty }4^{j-1}(CT^{\frac{1}{2}%
})^{j+1}\left\Vert R^{(k+j)}\gamma _{N,0}^{(k+j)}\right\Vert _{L_{\mathbf{x,x%
}^{\prime }}^{2}}
\end{eqnarray*}%
Use condition (\ref{condition:energy condition}), it becomes%
\begin{eqnarray*}
&\leqslant &CT^{\frac{1}{2}}C_{0}^{k}+\sum_{j=1}^{\infty }4^{j-1}(CT^{\frac{1%
}{2}})^{j+1}C_{0}^{k+j} \\
&\leqslant &C_{0}^{k}\left( CT^{\frac{1}{2}}+\sum_{j=1}^{\infty }4^{j-1}(CT^{%
\frac{1}{2}})^{j+1}C_{0}^{j}\right) .
\end{eqnarray*}%
We can choose a $T$ independent of $k$, $l_{c}$ and $N$ such that the series
in the above estimate converges. We then have%
\begin{eqnarray*}
\left\Vert R^{(k-1)}B_{N,1,k}\text{\textsc{FP}}^{(k,l_{c})}\right\Vert
_{L_{T}^{1}L_{\mathbf{x,x}^{\prime }}^{2}} &\leqslant &C_{0}^{k}\left( CT^{%
\frac{1}{2}}+C\right) \\
&\leqslant &C^{k-1}
\end{eqnarray*}%
for some $C$ larger than $C_{0}$. Whence, we have shown estimate %
\eqref{estimate:FreePartofBBGKY}.
\end{proof}

\begin{proof}[Proof of Estimate \eqref{estimate:InteractionPartofBBGKY}]
We proceed like the proof of estimate (\ref{estimate:FreePartofBBGKY}) and
end up with 
\begin{eqnarray*}
&&\left\Vert R^{(k-1)}B_{N,1,k}\QTR{sc}{IP}^{(k,l_{c})}\right\Vert
_{L_{T}^{1}L_{\mathbf{x,x}^{\prime }}^{2}} \\
&\leqslant &\sum_{m}\left\Vert R^{(k-1)}B_{N,1,k}\int_{D}J_{N}^{(k,l_{c}+1)}(%
\underline{t}_{k+l_{c}+1},\mu _{m})(\gamma
_{N}^{(k+l_{c}+1)}(t_{k+l_{c}+1}))d\underline{t}_{k+l_{c}+1}\right\Vert
_{L_{T}^{1}L_{\mathbf{x,x}^{\prime }}^{2}} \\
&\leqslant &\sum_{m}(CT^{\frac{1}{2}})^{l_{c}}\left\Vert
R^{(k+l_{c})}B_{N,\mu _{m}(k+l_{c}+1),k+l_{c}+1}\gamma
_{N}^{(k+l_{c}+1)}(t_{k+l_{c}+1})\right\Vert _{L_{T}^{1}L_{\mathbf{\ x,x}%
^{\prime }}^{2}}.
\end{eqnarray*}%
We then investigate 
\begin{equation*}
\left\Vert R^{(k+l_{c})}B_{N,\mu _{m}(k+l_{c}+1),k+l_{c}+1}\gamma
_{N}^{(k+l_{c}+1)}(t_{k+l_{c}+1})\right\Vert _{L_{T}^{1}L_{\mathbf{x,x}%
^{\prime }}^{2}}.
\end{equation*}%
Set $\mu _{m}(k+l_{c}+1)=1$ for simplicity and look at $\tilde{B}%
_{N,1,k+l_{c}+1}$, we have%
\begin{eqnarray*}
&&R^{(k+l_{c})}\tilde{B}_{N,1,k+l_{c}+1}\gamma _{N}^{(k+l_{c}+1)}(t) \\
&=&R^{(k+l_{c})}\int V_{N}(x_{1}-x_{k+l_{c}+1})\gamma _{N}^{(k+l_{c}+1)}(t,%
\mathbf{x}_{k+l_{c}},x_{k+l_{c}+1};\mathbf{x}_{k+l_{c}}^{\prime
},x_{k+l_{c}+1})dx_{k+l_{c}+1} \\
&=&I+II
\end{eqnarray*}%
with $I$ and $II$ given by the product rule:%
\begin{equation*}
I=\int V_{N}^{\prime }(x_{1}-x_{k+l_{c}+1})\left( \frac{R^{(k+l_{c})}}{%
\left\vert \nabla _{x_{1}}\right\vert }\right) \gamma _{N}^{(k+l_{c}+1)}(t,%
\mathbf{x}_{k+l_{c}},x_{k+l_{c}+1};\mathbf{x}_{k+l_{c}}^{\prime
},x_{k+l_{c}+1})dx_{k+l_{c}+1},
\end{equation*}%
\begin{equation*}
II=\int V_{N}(x_{1}-x_{k+l_{c}+1})\left( R^{(k+l_{c})}\gamma
_{N}^{(k+l_{c}+1)}\right) (t,\mathbf{x}_{k+l_{c}},x_{k+l_{c}+1};\mathbf{x}%
_{k+l_{c}}^{\prime },x_{k+l_{c}+1})dx_{k+l_{c}+1},
\end{equation*}%
where we wrote%
\begin{equation*}
\frac{R^{(k+l_{c})}}{\left\vert \nabla _{x_{1}}\right\vert }=\left(
\dprod_{j=2}^{k+l_{c}}\left\vert \nabla _{x_{j}}\right\vert \right) \left(
\dprod_{j=1}^{k+l_{c}}\left\vert \nabla _{x_{j}^{\prime }}\right\vert
\right) .
\end{equation*}%
Then 
\begin{eqnarray*}
&&\int \left\vert R^{(k+l_{c})}\tilde{B}_{N,1,k+l_{c}+1}\gamma
_{N}^{(k+l_{c}+1)}(t)\right\vert ^{2}d\mathbf{x}_{k+l_{c}}d\mathbf{x}%
_{k+l_{c}}^{\prime } \\
&=&\int \left\vert I+II\right\vert ^{2}d\mathbf{x}_{k+l_{c}}d\mathbf{x}%
_{k+l_{c}}^{\prime } \\
&\leqslant &C\int \left\vert I\right\vert ^{2}d\mathbf{x}_{k+l_{c}}d\mathbf{x%
}_{k+l_{c}}^{\prime }+C\int \left\vert II\right\vert ^{2}d\mathbf{x}%
_{k+l_{c}}d\mathbf{x}_{k+l_{c}}^{\prime }.
\end{eqnarray*}%
To estimate the first term, we first Cauchy-Schwarz $dx_{k+l_{c}+1}$,%
\begin{eqnarray*}
&&\int \left\vert I\right\vert ^{2}d\mathbf{x}_{k+l_{c}}d\mathbf{x}%
_{k+l_{c}}^{\prime } \\
&\leqslant &\int d\mathbf{x}_{k+l_{c}}d\mathbf{x}_{k+l_{c}}^{\prime }\left(
\int \left\vert V_{N}^{\prime }(x_{1}-x_{k+l_{c}+1})\right\vert
^{2}dx_{k+l_{c}+1}\right) \\
&&\times \left( \int \left\vert \left( \frac{R^{(k+l_{c})}}{\left\vert
\nabla _{x_{1}}\right\vert }\right) \gamma _{N}^{(k+l_{c}+1)}(t,\mathbf{x}%
_{k+l_{c}},x_{k+l_{c}+1};\mathbf{x}_{k+l_{c}}^{\prime
},x_{k+l_{c}+1})\right\vert ^{2}dx_{k+l_{c}+1}\right) \\
&\leqslant &N^{5\beta }\left\Vert V^{\prime }\right\Vert _{L^{2}}^{2}\int d%
\mathbf{x}_{k+l_{c}}d\mathbf{x}_{k+l_{c}}^{\prime } \\
&&\times \left( \int \left\vert S^{(k+l_{c})}\gamma _{N}^{(k+l_{c}+1)}(t,%
\mathbf{x}_{k+l_{c}},x_{k+l_{c}+1};\mathbf{x}_{k+l_{c}}^{\prime
},x_{k+l_{c}+1})\right\vert ^{2}dx_{k+l_{c}+1}\right)
\end{eqnarray*}%
where $V\in W^{2,\frac{6}{5}+}$ implies that $V\in H^{1}$ by Sobolev. A
trace theorem then gives%
\begin{eqnarray*}
&\leqslant &CN^{5\beta }\left\Vert V^{\prime }\right\Vert _{L^{2}}^{2}\int d%
\mathbf{x}_{k+l_{c}}d\mathbf{x}_{k+l_{c}}^{\prime } \\
&&\times \left( \int \left\vert S^{(k+l_{c}+1)}\gamma _{N}^{(k+l_{c}+1)}(t,%
\mathbf{x}_{k+l_{c}},x_{k+l_{c}+1};\mathbf{x}_{k+l_{c}}^{\prime
},x_{k+l_{c}+1}^{\prime })\right\vert
^{2}dx_{k+l_{c}+1}dx_{k+l_{c}+1}^{\prime }\right) \\
&=&CN^{5\beta }\left\Vert V^{\prime }\right\Vert _{L^{2}}^{2}\left\Vert
S^{(k+l_{c}+1)}\gamma _{N}^{(k+l_{c}+1)}\right\Vert _{L_{T}^{\infty }L_{%
\mathbf{x},\mathbf{x}^{\prime }}^{2}}^{2}.
\end{eqnarray*}%
Estimate the second term in the same manner, we get 
\begin{eqnarray*}
&&\int \left\vert II\right\vert ^{2}d\mathbf{x}_{k+l_{c}}d\mathbf{x}%
_{k+l_{c}}^{\prime } \\
&=&\int d\mathbf{x}_{k+l_{c}}d\mathbf{x}_{k+l_{c}}^{\prime } \Big\vert \int V_{N}(x_{1}-x_{k+l_{c}+1})\\
&&\times \left( R^{(k+l_{c})}\gamma
_{N}^{(k+l_{c}+1)}\right) (t,\mathbf{x}_{k+l_{c}},x_{k+l_{c}+1};\mathbf{x}%
_{k+l_{c}}^{\prime },x_{k+l_{c}+1})dx_{k+l_{c}+1}\Big\vert ^{2} \\
&\leqslant &CN^{3\beta }\left\Vert V\right\Vert _{L^{2}}^{2}\left\Vert
S^{(k+l_{c}+1)}\gamma _{N}^{(k+l_{c}+1)}\right\Vert _{L_{T}^{\infty }L_{%
\mathbf{x},\mathbf{x}^{\prime }}^{2}}^{2},
\end{eqnarray*}%
Accordingly, 
\begin{equation*}
\begin{aligned}
\indentalign \int \left\vert R^{(k+l_{c})}\tilde{B}_{N,1,k+l_{c}+1}\gamma
_{N}^{(k+l_{c}+1)}(t)\right\vert ^{2}d\mathbf{x}_{k+l_{c}}d\mathbf{x}%
_{k+l_{c}}^{\prime } \\
&\leqslant CN^{5\beta }\left\Vert V\right\Vert
_{H^{1}}^{2}\left\Vert S^{(k+l_{c}+1)}\gamma _{N}^{(k+l_{c}+1)}\right\Vert
_{L_{T}^{\infty }L_{\mathbf{x},\mathbf{x}^{\prime }}^{2}}^{2}.
\end{aligned}
\end{equation*}%
Thence 
\begin{eqnarray*}
&&\left\Vert R^{(k-1)}B_{N,1,k}\QTR{sc}{IP}^{(k,l_{c})}\right\Vert
_{L_{T}^{1}L_{\mathbf{x,x}^{\prime }}^{2}} \\
&\leqslant &\sum_{m}(CT^{\frac{1}{2}})^{l_{c}}\left\Vert
R^{(k+l_{c})}B_{N,\mu _{m}(k+l_{c}+1),k+l_{c}+1}\gamma
_{N}^{(k+l_{c}+1)}(t_{k+l_{c}+1})\right\Vert _{L_{T}^{1}L_{\mathbf{\ x,x}%
^{\prime }}^{2}} \\
&\leqslant &C4^{l_{c}}(CT^{\frac{1}{2}})^{l_{c}}T\left( CN^{\frac{5\beta }{2}%
}\left\Vert V\right\Vert _{H^{1}}\left\Vert S^{(k+l_{c}+1)}\gamma
_{N}^{(k+l_{c}+1)}\right\Vert _{L_{T}^{\infty }L_{\mathbf{x},\mathbf{x}%
^{\prime }}^{2}}\right)
\end{eqnarray*}%
Put in Condition (\ref{condition:energy condition}), it becomes%
\begin{eqnarray*}
\left\Vert R^{(k-1)}B_{N,1,k}\QTR{sc}{IP}^{(k,l_{c})}\right\Vert
_{L_{T}^{1}L_{\mathbf{x,x}^{\prime }}^{2}} &\leqslant &CT(CT^{\frac{1}{2}%
})^{l_{c}}N^{\frac{5\beta }{2}}C_{0}^{k+l_{c}+1} \\
&=&C_{0}^{k}\left[ CT(CT^{\frac{1}{2}})^{l_{c}}N^{\frac{5\beta }{2}%
}C_{0}^{l_{c}+1}\right] .
\end{eqnarray*}%
Replace the constants $C$ and $C_{0}$ inside the bracket with some larger
constant and group the terms, we have%
\begin{equation*}
\left\Vert R^{(k-1)}B_{N,1,k}\QTR{sc}{IP}^{(k,l_{c})}\right\Vert
_{L_{T}^{1}L_{\mathbf{x,x}^{\prime }}^{2}}\leqslant C_{0}^{k}\left[ (T^{%
\frac{1}{2}})^{2+l_{c}}N^{\frac{5\beta }{2}}C^{l_{c}}\right] .
\end{equation*}%
As in \cite{TChenAndNPSpace-Time,Chen3DDerivation}, we take the coupling
level $l_{c}=\ln N$ to deal with what is inside the bracket 
\begin{equation*}
\left\Vert R^{(k-1)}B_{N,1,k}\QTR{sc}{IP}^{(k,l_{c})}\right\Vert
_{L_{T}^{1}L_{\mathbf{x,x}^{\prime }}^{2}}\leqslant CC_{0}^{k}\left[ (T^{%
\frac{1}{2}})^{2+\ln N}N^{\frac{5\beta }{2}}N^{c}\right] .
\end{equation*}%
Notice that there is no $k$ inside the bracket. Selecting $T$ such that 
\begin{equation*}
T\leqslant e^{-(5\beta +2C)},
\end{equation*}%
ensures that 
\begin{equation*}
(T^{\frac{1}{2}})^{\ln N}N^{\frac{5\beta }{2}}N^{c}\leqslant 1,
\end{equation*}%
and thence 
\begin{equation*}
\left\Vert R^{(k-1)}B_{N,1,k}\QTR{sc}{IP}^{(k,l_{c})}\right\Vert
_{L_{T}^{1}L_{\mathbf{x,x}^{\prime }}^{2}}\leqslant CC_{0}^{k}\leqslant
C^{k-1}
\end{equation*}%
with a $C$ larger than $C_{0}$ and independent of $k$ and $N.$ Whence, we
have finished the proof of estimate (\ref{estimate:InteractionPartofBBGKY}).
\end{proof}

\end{document}